\documentclass{article}
\usepackage{geometry}[top=15truemm, bottom=15truemm, right=10truemm, left=10truemm]
\usepackage{amsmath,amssymb,amsthm}
\usepackage{amsmath,amssymb,amsthm}
\usepackage[all]{xy}
\usepackage{graphicx}
\usepackage{mathrsfs}
\usepackage{lineno}

\newcommand{\fsets}{{\bf Sets}_f}

\newcommand{\intwitt}{W_{O_K}^a(O_{\bar{K}})}

\newcommand{\dfa}{\mathfrak{A}}

\newcommand{\C}{\mathscr{C}}

\newcommand{\modular}{\mathfrak{F}}

\newcommand{\cl}{\mathscr{B}_f}

\newcommand{\idp}{\mathfrak{p}}
\newcommand{\idq}{\mathfrak{q}}

\newcommand{\idf}{\mathfrak{f}}

\newcommand{\ida}{\mathfrak{a}}
\newcommand{\idb}{\mathfrak{b}}
\newcommand{\idc}{\mathfrak{c}}

\newcommand{\idP}{\mathfrak{P}}

\newcommand{\ratf}{\mathbb{Q}}
\newcommand{\comp}{\mathbb{C}}
\newcommand{\fld}{\mathbb{F}}
\newcommand{\realf}{\mathbb{R}}
\newcommand{\upper}{\mathfrak{H}}
\newcommand{\Ad}{\mathbb{A}}

\newcommand{\F}{\mathrm{F}}

\newcommand{\End}{\mathrm{End}}

\newcommand{\Hom}{\mathrm{Hom}}

\newcommand{\id}{\mathrm{id}}

\newcommand{\integer}{\mathbb{Z}}

\newcommand{\nat}{\mathbb{N}}
\newcommand{\J}{\mathcal{J}}
\newcommand{\idd}{\mathfrak{d}}

\theoremstyle{plain}
\newtheorem{thm}{Theorem}[subsection]
\newtheorem{lem}[thm]{Lemma}
\newtheorem{prop}[thm]{Proposition}
\newtheorem{cor}[thm]{Corollary}

\theoremstyle{definition}
\newtheorem{defn}[thm]{Definition}
\newtheorem{ex}[thm]{Example}

\newtheorem{rem}[thm]{Remark}

\numberwithin{equation}{section}
\title{Semi-galois Categories III:\\ Witt vectors by deformations of modular functions}
\author{Takeo Uramoto\\ Institute of Mathematics for Industry, Kyushu University}
\date{}
\begin{document}
\maketitle
\begin{abstract}
\noindent
Based on our previous work on an arithmetic analogue of Christol's theorem, this paper studies in more detail the structure of the $\Lambda$-ring $E_K = K \otimes \intwitt$ of algebraic Witt vectors for number fields $K$. First developing general results concerning $E_K$, we apply them to the case when $K$ is an imaginary quadratic field. The main results include the ``\emph{modularity theorem}'' for algebraic Witt vectors, which claims that certain deformation families $f: M_2(\widehat{\integer}) \times \upper \rightarrow \comp$ of modular functions of finite level always define algebraic Witt vectors $\widehat{f}$ by their special values, and conversely, every algebraic Witt vector $\xi \in E_K$ is realized in this way, that is, $\xi = \widehat{f}$ for some deformation family $f: M_2(\widehat{\integer}) \times \upper \rightarrow \comp$. This gives a rather explicit description of the $\Lambda$-ring $E_K$ for imaginary quadratic fields $K$, which is stated as the identity $E_K = M_K$ between the $\Lambda$-ring $E_K$ and the $K$-algebra $M_K$ of \emph{modular vectors} $\widehat{f}$. 
\end{abstract}
\section{Introduction}
\label{s1}
This paper is a continuation of our previous work on \emph{arithmetic analogue of Christol's theorem} \cite{Uramoto18}. This theorem claims that a (generalized) \emph{Witt vector} $\xi \in W_{O_K}(O_{\bar{K}})$ (\S 2 \cite{Uramoto18}) is integral over the ring $O_K$ of integers in a number field $K$ if and only if the orbit of $\xi$ under the action of the Frobenius lifts $\psi_\idp: W_{O_K}(O_{\bar{K}}) \rightarrow W_{O_K}(O_{\bar{K}}) $ is finite (cf.\ Theorem 3.4, \cite{Uramoto18}); we then deduced that, with the aid of the work of Borger and de Smit \cite{Borger_Smit11}, which heavily relies on class field theory, this is also precisely when the ghost components $\xi_\ida$ of $\xi$ are \emph{periodic} with respect to some modulus $\idf$ of $K$ (cf.\ Corollary 2, \cite{Uramoto18}). With this background, the major goal of the current paper is then to study in more detail the structure of the $\Lambda$-ring $E_K := K \otimes \intwitt$ (where $\intwitt$ is the $\Lambda$-ring of integral Witt vectors) in the case where $K$ is an imaginary quadratic field; in particular, we prove that the $\Lambda$-ring $E_K$ coincides as a $K$-subalgebra of $(K^{ab})^{I_K}$ with the $K$-algebra $M_K$ of \emph{modular vectors}--- i.e.\ those vectors $\widehat{f} \in (K^{ab})^{I_K}$ whose components $\widehat{f}_\ida$ are given by special values of certain deformation families $f: M_2(\widehat{\integer}) \times \upper \rightarrow \comp$ of modular functions--- i.e.\ the fiber $f_m := f(m, -): \upper \rightarrow \comp$ at each $m \in M_2(\widehat{\integer})$ is a modular function of finite level, and these $f_m$ satisfy certain correlation--- prototypical examples of such deformation families of modular functions are given by \emph{Fricke functions} $f_a$ ($a \in \ratf^2/\integer^2$); cf.\ \S \ref{s4.2}. In summary, the following is our major result, which we call the \emph{modularity theorem} (cf.\ \S \ref{s4.3}, \S \ref{s4.4}): 

\begin{thm}[modularity theorem]
We have the following identity as $K$-subalgebras of $(K^{ab})^{I_K}$: 
\begin{eqnarray}
 E_K &=& M_K. 
\end{eqnarray}
\end{thm}

\paragraph{Background}
To be precise, for a flat $O_K$-algebra $A$, recall that the ring $W_{O_K}(A)$ of (generalized) \emph{Witt vectors} with coefficients in $A$ \cite{Borger} is defined as the intersection $W_{O_K}(A) := \bigcap_n U_n(A)$ of the following $O_K$-algebras $U_n(A) \subseteq A^{I_K}$ given by induction on $n \geq 0$ (cf.\ \S 2.1 \cite{Uramoto18}; also see Remark 5, \S 2.2 \cite{Uramoto18} for some intuition on Witt vectors as ``smooth'' functions on $I_K$ with respect to \emph{arithmetic derivations} in the sense of Buium \cite{Buium}):
\begin{eqnarray}
 U_0 (A) &:=& A^{I_K}; \\
 U_{n+1} (A) &:=& \bigl\{ \xi \in U_n(A) \mid \forall \idp \in P_K. \hspace{0.1cm} \psi_\idp \xi - \xi^{N\idp} \in \idp U_n(A) \bigr\}; 
\end{eqnarray}
where $P_K$ and $I_K$ denote the set of maximal ideals of $O_K$ and the monoid of nonzero ideals of $O_K$ respectively; $A^{I_K}$ denotes the $I_K$-times product of $A$; $N\idp$ denotes the absolute norm of $\idp \in P_K$; and $\psi_\idp: A^{I_K} \rightarrow A^{I_K}$ denotes the shift $(\xi_\ida) \mapsto (\xi_{\idp \ida})$. In the case when $K$ is the rational number field $\ratf$, say, this ring $W_{O_K}(A)$ is isomorphic to the usual ring $W_\integer(A)$ of \emph{big Witt vectors}; and the usual ring $W_p(A)$ of \emph{$p$-typical Witt vectors} can be constructed in a similar way. While these rings of big and $p$-typical Witt vectors are conventionally constructed using Witt polynomials, Borger \cite{Borger} recasted these constructions of Witt vectors putting his focus on Frobenius lifts, and constructed the ring of Witt vectors as the universal ring among those rings which are equipped with commuting family of Frobenius lifts, or \emph{$\Lambda$-rings}; this construction naturally allows us to extend the base ring from $\integer$ to arbitrary Dedekind domains $O$ with finite residue fields (or more): He proved the existence of such a universal ring for this generalized setting by the above inductive construction. The basic theory of these generalized Witt vectors was developed in \cite{Borger}; our major concern in this paper is to study the structure of the rings of these generalized Witt vectors.

In particular, among generalized Witt vectors $\xi \in W_{O_K}(O_{\bar{K}})$ with coefficients in the ring $O_{\bar{K}}$ of algebraic integers, those $\xi \in W_{O_K}(O_{\bar{K}})$ which are integral over $O_K$--- i.e.\ \emph{integral Witt vectors}--- were proved to be relevant to class field theory of the number field $K$, as  discussed in our previous work \cite{Uramoto18} based on \cite{Borger_Smit08,Borger_Smit11}; see also \cite{Borger_Smit18}. In fact, on the one hand, it was proved in \cite{Borger_Smit11} that the category $\C_K$ of those (generalized) $\Lambda$-rings which are finite etale over $K$ and have \emph{integral models} (\S 2 \cite{Uramoto18}) is dually equivalent to the category $\cl DR_K$ of finite $DR_K$-sets, where $DR_K$ is the profinite monoid called the \emph{Deligne-Ribet monoid} and given by inverse limit of ray class monoids $DR_\idf$ ($\idf \in I_K$); in this proof, class field theory was used in an essential way, which suggests an inherent connection between class field theory and (generalized) $\Lambda$-rings. Motivated by this duality \cite{Borger_Smit11}, our previous work \cite{Uramoto18} then related integral Witt vectors to the objects of $\C_K$; to be precise, we proved that for any integral Witt vector $\xi \in \intwitt$, the $\Lambda$-ring $X_\xi = K \otimes O_K \langle \xi \rangle$ generated by the orbit $I_K \xi$ of $\xi$ under the action of the Frobenius lifts $\psi_\idp$ is finite etale over $K$ and has an integral model, or in other words, forms an object of $\C_K$ (cf.\ Proposition 4 \cite{Uramoto18})--- indeed as further proved in this paper, these types of objects $X_\xi$ are universal in $\C_K$ (cf.\ \S \ref{s3}). Technically speaking, our major result there, i.e.\ an arithmetic analogue of Christol's theorem (Theorem 3.4 \cite{Uramoto18}), is necessary to prove the finiteness of $X_\xi$ over $K$. (However, this theorem itself is of independent interest in that it gives a natural arithmetic (or $\fld_1$-) analogue of Christol's theorem \cite{Christol,CKFG} on formal power series $\xi \in \fld_q[[t]]$ over finite field $\fld_q$; cf.\ \S 3 \cite{Uramoto18}.)

\paragraph{Contribution}
As a natural continuation of \cite{Uramoto18}, the current paper then studies in more detail the structure of the ring $\intwitt$ of integral Witt vectors and the $K$-algebra $E_K := K \otimes \intwitt$ whose elements we call \emph{algebraic} Witt vectors. In particular, after developing some general basic results about the $K$-algebra $E_K$ (\S \ref{s3}), we then apply them to the case when $K$ is the rational number field $\ratf$ (\S \ref{s3.3}) and an imaginary quadratic field (\S \ref{s4}). The major results of this paper are those given in \S \ref{s4}, where, as briefly summarized above, we relate algebraic Witt vectors $\xi \in E_K$ with certain deformation families of modular functions. 

To be more specific, the subject of \S \ref{s3} is to classify and study the structure of the \emph{galois objects} of $\C_K$ (cf.\ \S 3.2 \cite{Uramoto17}), which provides a general basis for the study in \S \ref{s4}. In particular, we see that, for each finite set $\Xi \subseteq \intwitt$ of integral Witt vectors, we can construct a $\Lambda$-ring $X_\Xi$, which is an object of $\C_K$ as in the case of singleton $\Xi = \{\xi\}$ mentioned above. Concerning this construction, we prove in \S \ref{s3.1} that $X_\Xi$ is a galois object of $\C_K$ for every $\Xi \subseteq \intwitt$, and conversely, every galois object of $\C_K$ is of this form up to isomorphism. That is, the galois objects of $\C_K$ are precisely those of the form $X_\Xi$ for some finite $\Xi \subseteq \intwitt$. This particularly implies that the direct limit of the galois objects of $\C_K$ is naturally isomorphic to our $K$-algebra $E_K = K \otimes \intwitt$; this gives a characterization of the $K$-algebra $E_K$ of algebraic Witt vectors as the universal ring among those $\Lambda$-rings finite etale over $K$ and having integral models (i.e.\ universal with respect to the objects of $\C_K$). For the purpose of \S \ref{s4}, the rest of \S \ref{s3} (i.e.\ \S \ref{s3.2} and \S \ref{s3.3}) is somewhat optional, but the results there will be of independent interest. In \S \ref{s3.2}, on the one hand, we study the structure of the galois objects of $\C_K$; by the result in \S \ref{s3.1}, this means that we study the $\Lambda$-rings of the form $X_\Xi$ for some $\Xi \subseteq \intwitt$. In this subsection, we particularly determine the \emph{state complexity} of integral Witt vectors $\xi \in \intwitt$, i.e.\ the minimum size $c_\xi$ of \emph{deterministic finite automata} (DFAs) that generate $\xi$ (cf.\ \S 3.1 \cite{Uramoto18}); we prove that the state complexity $c_\xi$ for $\xi \in \intwitt$ is equal to the dimension $\dim_K X_\xi$ of the $\Lambda$-ring $X_\xi$ over $K$. (This is an $\fld_1$-analogue of Bridy's result \cite{Bridy} on formal power series $\xi \in \fld_q[[t]]$; cf.\ \S \ref{s3.2}.) In \S \ref{s3.3}, on the other hand, we determine the structure of $\intwitt$ and $E_K$ for the case when $K$ is the rational number field $\ratf$ as an immediate consequence of the result in \S \ref{s3.1} and the result in \cite{Borger_Smit08}: we show that the $\Lambda$-ring $W_\integer^a(\bar{\integer})$ of integral Witt vectors is isomorphic to the group-ring $\integer[\ratf/\integer]$, hence, $E_\ratf$ is isomorphic to the group-algebra $\ratf[\ratf/\integer]$. In more elementary words, this isomorphism $E_\ratf \simeq \ratf[\ratf/\integer]$ shows that the algebraic (resp.\ integral) Witt vectors $\xi = (\xi_n)_{n \in \nat} \in E_\ratf$ are precisely the $\ratf$-linear (resp.\ $\integer$-linear) combinations of the vectors $\zeta^{(\gamma)} \in (\ratf^{ab})^{\nat}$ of the form $\zeta^{(\gamma)} = (e^{2 \pi i \gamma n})_n$ for $\gamma \in \ratf/\integer$. 

The subject of \S \ref{s4} is then to proceed this study to the case where $K$ is an imaginary quadratic field. To be more specific, we prove in \S \ref{s4.3} that the algebraic Witt vectors $\xi \in E_K$ are precisely the \emph{modular vectors}, i.e.\ the vectors $\widehat{f} \in (K^{ab})^{I_K}$ whose components $\widehat{f}_\ida$ are defined by special values of certain deformation families $f: M_2(\widehat{\integer}) \times \upper \rightarrow \comp$ of modular functions, which we shall call \emph{Witt deformation families of modular functions} (because their special values eventually define algebraic Witt vectors). Briefly, a Witt deformation (family) of modular functions is defined as a continuous function $f: M_2(\widehat{\integer}) \times \upper \rightarrow \comp$ such that the fiber $f_m := f(m, -): \upper \rightarrow \comp$ at each $m \in M_2(\widehat{\integer})$ is a modular function of finite level, where the fibers $f_m$ satisfy a few conditions (cf.\ Definition \ref{Witt deformation family}, \S \ref{s4.2}), including e.g.\ the property that $f_{m \gamma} (\gamma^{-1}\tau) = f_m (\tau)$ for $\gamma \in SL_2(\integer) = \Gamma$ in particular. Concerning this, we see that each Witt deformation $f: M_2(\widehat{\integer}) \times \upper \rightarrow \comp$ of modular functions defines a modular vector $\widehat{f} \in (K^{ab})^{I_K}$ in a natural (but non-trivial) way that heavily relies on the work of Connes, Marcolli and Ramachandran \cite{Connes_Marcolli_Ramachandran}, Connes and Marcolli \cite{Connes_Marcolli}, Laca, Larsen, and Neshveyev \cite{Laca_Larsen_Neshveyev} and Yalkinoglu \cite{Yalkinoglu}. Briefly speaking, this construction $f \mapsto \widehat{f}$ is based on the isomorphisms between the above-mentioned Deligne-Ribet monoid $DR_K$ and the (moduli) space $Lat_K^1$ of \emph{$1$-dimensional $K$-lattices} \cite{Connes_Marcolli_Ramachandran,Laca_Larsen_Neshveyev,Yalkinoglu}, and between the (moduli) space $Lat^2_\ratf$ of \emph{$2$-dimensional $\ratf$-lattices} and the quotient space $\Gamma \backslash (M_2(\widehat{\integer}) \times \upper)$ \cite{Connes_Marcolli} (cf.\ \S \ref{s2.2}). Since $K$ is now imaginary quadratic, hence, of degree $2$ over $\ratf$, we have a natural embedding $Lat_K^1 \hookrightarrow Lat_\ratf^2$;  and then composed with the isomorphisms $DR_K \simeq Lat_K^1$ and $Lat_\ratf^2 \simeq \Gamma \backslash (M_2(\widehat{\integer}) \times \upper)$, as well as the canonical embedding $I_K \hookrightarrow DR_K$, each Witt deformation $f: M_2(\widehat{\integer}) \times \upper \rightarrow \comp$ then induces a function $\widehat{f}: I_K \rightarrow \comp$ as follows:
\begin{equation}
 \widehat{f} : I_K \hookrightarrow DR_K \xrightarrow{\simeq} Lat_K^1 \hookrightarrow Lat_\ratf^2 \xrightarrow{\simeq} \Gamma \backslash (M_2(\widehat{\integer}) \times \upper) \xrightarrow{f} \comp. 
\end{equation}
We can see that $\widehat{f}$ takes its values in $K^{ab}$, hence defines a vector $\widehat{f} \in (K^{ab})^{I_K}$ which is what we call the \emph{modular vector associated to $f$}; the $K$-subalgebra $M_K \subseteq (K^{ab})^{I_K}$ is defined as that consisting of modular vectors in this sense. The major result of \S \ref{s4} is the ``\emph{modularity theorem}'' claiming the identity $E_K = M_K$ as $K$-subalgebras of $(K^{ab})^{I_K}$ (\S \ref{s4.3}). In other words, for every Witt deformation $f: M_2(\widehat{\integer}) \times \upper \rightarrow \comp$, the associated modular vector $\widehat{f}$ defies an algebraic Witt vector, i.e.\ $\widehat{f} \in E_K$; and conversely, every algebraic Witt vector $\xi \in E_K$ is realized in this way, i.e.\ $\xi = \widehat{f}$ for some Witt deformation $f: M_2(\widehat{\integer}) \times \upper \rightarrow \comp$. This identity $E_K = M_K$ relates the geometry of $\Lambda$-rings in $\C_K$ and that of modular functions. Moreover, in the last subsection (\S \ref{s4.4}) we prove that $E_K = M_K$ is generated as a $K$-algebra by the modular vectors $\widehat{f}_a$ naturally arising from the Fricke functions $f_a$ ($a \in \ratf^2/\integer^2$).  


\paragraph{Acknowledgement}
We are grateful to James Borger for fruitful discussions, which motivated us to prove Proposition \ref{galois is X_Xi} instead of considering the existence of cyclic Witt vectors, cf.\ Remark \ref{cyclic witt vector}; and to Naoya Yamanaka and Hayato Saigo for their support and encouragement. We are also thankful to an anonymous reviewer for his/her patient proofreading, in particular, alerting us to an error in Lemma 4.4.2 in \cite{Uramoto20v4}; to Yasuhiro Ishitsuka for alerting us to an error in Proposition 3 \cite{Uramoto18} (cf.\ corrigendum given in \S \ref{s2}). This work was done during the author was a member of Nagahama Institute for Bio-Science and Technology, to whom we are grateful for their hospitality. This work was supported by JSPS KAKENHI Grant number JP16K21115. 

\section{Preliminaries}
\label{s2}
\noindent
This section summarizes necessary terminology and results from \cite{Uramoto18,Connes_Marcolli,Connes_Marcolli_Ramachandran,Laca_Larsen_Neshveyev,Yalkinoglu}. But in order to avoid duplications, we refer  the reader to \S 2 -- \S 4 \cite{Uramoto18} for the detailed definitions of basic concepts and results in \cite{Uramoto18}, say, \emph{Witt vectors} (Definition 1, pp.\ 543), \emph{$\Lambda$-rings} (Definition 2, pp.\ 544), \emph{integral models of finite etale $\Lambda$-rings over $K$} (Remark 4, pp.\ 544) and our \emph{arithmetic analogue of Christol's theorem} (Theorem 3.4, pp.\ 557) in particular. Also, for the categorical concepts and results on \emph{semi-galois categories}, we refer the reader to \cite{Uramoto17}. 

But, before proceeding to the main part, let us include here the following corrigendum to \cite{Uramoto18}:

\paragraph{Corrigendum to \cite{Uramoto18}}
Yasuhiro Ishitsuka alerted us that the proof of Proposition 3, \S 3.2, pp.552 \cite{Uramoto18} is valid only for the case when $K$ is of class number one; to be specific, an error occurs in the last part of the proof claiming that $b/\pi \in O_{K, \idq}$ for $\idp \neq \idq$, which is invalid for our choice of $\pi \in \idp$ in general because, when $\idp$ is not principal, $\pi \in \idp$ would belong in some $\idq \neq \idp$ too. We are thankful to him for pointing this out. Thus  we restrict Proposition 3 \cite{Uramoto18} only to the case when $K$ is of class number one.

Accordingly, we repair here our proof of the main result of \cite{Uramoto18}, which depended on Proposition 3. 
Note that Proposition 3 was used in the proofs of Lemma 3.2 (pp.554) and Lemma 3.4 (pp.555). However, in the proof of Lemma 3.2, on the one hand, we actually need only Lemma 3.1. On the other hand, Lemma 3.4 could be also proved without Proposition 3 as follows\footnote{This salvage of Lemma 3.4 is optional in that, as discussed in \cite{Uramoto20v6}, the full statement of Lemma 3.4 was in fact unnecessary for the current paper and \cite{Uramoto18}. Nevertheless we include it here for future reference.}: (The notations are basically as in \cite{Uramoto18}; but in order to indicate the base rings explicitly, we shall denote $U_n(A)$ of \S 2.1 \cite{Uramoto18} as $U_{O_K, n} (A)$.)
\begin{lem}[refinement of Lemma 3.4 \cite{Uramoto18}]
Let $L/K$ be a finite (galois) extension. For any non-negative integer $n \geq 0$, if $\xi \in U_{O_K, n}(O_L)$ then $N^* \xi \in U_{O_L, n}(O_L)$. In particular, if $\xi \in W_{O_K}(O_L)$, then $N^* \xi \in W_{O_L} (O_L)$; and if  $\xi \in W_{O_K}^a(O_L)$, then $N^* \xi \in W_{O_L}^a (O_L)$. 
\end{lem}
\begin{proof}
Since the first claim implies the others, we prove only the first claim by induction on $n \geq 0$. For the base case $n=0$, there is nothing to prove. For induction, assume that the claim is true up to $n$, with which we show that $N^* \xi \in U_{O_L, n+1}(O_L)$ for $\xi \in U_{O_K, n+1}(O_L)$. Take $\xi \in U_{O_K, n+1}(O_L)$ and let us put $\zeta := N^* \xi \in O_L^{I_L}$. 

First, by $U_{O_K, n+1}(O_L) \subseteq U_{O_K, n}(O_L)$ and induction hypothesis, we indeed have $\zeta \in U_{O_K, n}(O_L)$; then, to show $\zeta \in U_{O_L, n+1}(O_L)$, we need prove that $\psi_\idP \zeta - \zeta^{N\idP} \in \idP U_{O_L, n}(O_L)$ for every $\idP \in P_L$. To this end we note that, by $\xi \in U_{O_K, n+1}(O_L)$, we have $\psi_\idp \xi - \xi^{N\idp} \in \idp U_{O_K, n}(O_L)$ for $\idp := \idP \cap O_K$, from which we deduce $\psi_{\idp^f} \xi - \xi^{(N\idp)^f} \in \idp U_{O_K, n}(O_L)$ where $f$ denotes the inertia degree for $\idP \mid \idp$. Therefore, we have some $r_i \in \idp$ and $\eta^{(i)} \in U_{O_K, n}(O_L)$ such that:
\begin{eqnarray}
\psi_{\idp^f} \xi - \xi^{(N\idp)^f} &=& \sum_{i=1}^m r_i \cdot \eta^{(i)}. 
\end{eqnarray}
Applying $N^*$ to the both side of this equality, and noting that $N^* (\psi_{\idp^f} \xi) = \psi_\idP (N^* \xi) = \psi_\idP \zeta$ and $(N\idp)^f = N\idP$, we obtain:
\begin{eqnarray}
 \psi_\idP \zeta - \zeta^{N\idP} &=& \sum_{i=1}^m r_i \cdot N^* \eta^{(i)}.
\end{eqnarray}
By induction hypothesis applied to $\eta^{(i)} \in U_{O_K, n}(O_L)$, we know that $N^* \eta^{(i)} \in U_{O_L, n}(O_L)$; also by $r_i \in \idp \subseteq \idP$, we then obtain $\psi_\idP \zeta - \zeta^{N\idP} \in \idP U_{O_L, n}(O_L)$, which is what we needed to prove. 
\end{proof}

\noindent
With this modification, our main result (i.e.\ Theorem 3.4) of \cite{Uramoto18} remains valid. 

\subsection{The category $\C_K$}
\label{s2.1}
Here we recall some basic facts concerning the semi-galois category $\C_K$ constructed by Borger and de Smit \cite{Borger_Smit11} for a number field $K$. Let $K$ be a number field and $O_K$ the ring of integers in $K$. We denote by $P_K$ and $I_K$ the set of non-zero prime ideals of $O_K$ and the monoid of non-zero ideals of $O_K$ respectively. For each $\idp \in P_K$, we denote by $k_\idp$ the residue field $k_\idp := O_K/ \idp$ and by $N\idp$ the absolute norm $N \idp := \# k_\idp$. 

\begin{defn}[the category $\C_K$]
The objects of the category $\C_K$ \cite{Borger_Smit11} are the $\Lambda$-rings that are finite etale over $K$ and have integral models (cf.\ Remark 4, pp.544 \cite{Uramoto18}); the arrows of $\C_K$ are the $\Lambda$-ring homomorphisms over $K$ between them. 
\end{defn}

The category $\C_K$ is equipped with a functor $\F_K: \C_K^{op} \rightarrow \fsets$ that assigns to each object $X \in \C_K$ the finite set $\F_K(X) := \Hom_K (X, \bar{K})$ of the $K$-algebra homomorphisms from the underlying $K$-algebra $X$ to the algebraic closure $\bar{K}$ of $K$; and to each arrow $f: X \rightarrow Y$ of $\Lambda$-rings the pullback $f^*: \Hom_K(Y, \bar{K}) \rightarrow \Hom_K(X, \bar{K})$. With this functor $\F_K$, we can see that the pair $\langle \C_K^{op}, \F_K \rangle$ forms a semi-galois category. Therefore, by the general results developed in \cite{Uramoto17}, the semi-galois category $\langle \C_K^{op}, \F_K \rangle$ should be canonically equivalent to the semi-galois category $\langle \cl M, \F_M \rangle$ of finite $M$-sets with $M$ being the fundamental monoid $M = \pi_1(\C_K^{op}, \F_K)$ (cf.\ Definition 4, \S 2.2 \cite{Uramoto17}). 

But in the current situation, Borger and de Smit \cite{Borger_Smit11} could give a more explicit description of this equivalence. To be specific, recall that by definition an object $X \in \C_K$ is a finite etale $\Lambda$-ring over $K$, i.e., a finite etale $K$-algebra equipped with commuting family of Frobenius lifts $\psi_\idp: X \rightarrow X$ for each $\idp \in P_K$; thus, the finite set $\F_K (X) = \Hom_K (X, \bar{K})$ is equipped with a continuous action of the absolute galois group $G_K$ of $K$ \emph{as well as} the action of the monoid $I_K$ by natural pullbacks of $\psi_\ida: X \rightarrow X$ for $\ida \in I_K$ (see Remark 2, \S 2.2  \cite{Uramoto18} for this notation $\psi_\ida$); and the actions of $G_K$ and $I_K$ commute with each other. Therefore, the finite set $\F_K (X)$ forms a finite $(G_K \times I_K)$-set with respect to this action. To be specific, the action of the monoid $G_K \times I_K$ on the finite set $\F_K(X)$ is given as follows: 
\begin{eqnarray*}
 \F_K(X) \times (G_K \times I_K)  &\rightarrow& \F_K(X) \\
 (s, (\sigma, \ida)) &\longmapsto& \sigma \circ s \circ \psi_\ida. 
\end{eqnarray*}
Of course, the objects of $\C_K$ are not just finite etale $\Lambda$-rings over $K$ but have \emph{integral models} (cf.\ Remark 4, \S 2.2 \cite{Uramoto18}); Borger and de Smit \cite{Borger_Smit11} gave a characterization of when a finite etale $\Lambda$-ring $X$ has an integral model in terms of the corresponding finite $(G_K \times I_K)$-set $\F_K (X)$. 
As proved there, the \emph{Deligne-Ribet monoid} $DR_K$ (after their work \cite{Deligne_Ribet} where this profinite monoid appeared) comes into play for this characterization: As we recall below, we have a canonical map $G_K \times I_K \rightarrow DR_K$; and it was proved in \cite{Borger_Smit11} that a finite etale $\Lambda$-ring $X$ over $K$ has an integral model (i.e.\ is an object of $\C_K$) if and only if the $(G_K \times I_K)$-action on $\F_K(X)$ factors through this map $G_K \times I_K \rightarrow DR_K$ (Theorem 1.2 \cite{Borger_Smit11}). 

To be precise, the Deligne-Ribet monoid $DR_K$ is defined as the inverse limit of the \emph{ray class monoids}, i.e.\ the finite monoids $DR_\idf$ given for each $\idf \in I_K$ as the quotient monoid $DR_\idf := I_K / \sim_\idf$ of the monoid $I_K$ by the following congruence relation $\sim_\idf$ on $I_K$: for $\ida, \idb \in I_K$, 
\begin{eqnarray}
\label{ideal congruence}
 \ida \sim_\idf \idb &\Leftrightarrow& \ida \idb^{-1} = (t) \hspace{0.2cm} \textrm{for $\exists t \in K_+ \cap (1 + \idf \idb^{-1})$};
\end{eqnarray}
where $K_+$ denotes the totally positive elements of $K$. This congruence relation $\sim_\idf$ is of finite index, and thus, the ray class monoid $DR_\idf$ is a finite monoid. To form an inverse system of the ray class monoids $DR_\idf$, note that if $\idf \mid \idf'$ then $\ida \sim_{\idf'} \idb$ implies $\ida \sim_\idf \idb$; therefore, we have a canonical monoid surjection $DR_{\idf'} \twoheadrightarrow DR_\idf$ by the assignment $[\ida]_{\idf'} \mapsto [\ida]_\idf$ where $[\ida]_\idf \in DR_\idf$ denotes the equivalence class of $\ida \in I_K$ in $DR_\idf$. With respect to these surjections $DR_{\idf'} \twoheadrightarrow DR_\idf$, the ray class monoids $DR_\idf$ ($\idf \in I_K$) constitute an inverse system of finite monoids; and the Deligne-Ribet monoid $DR_K$ is then defined as the inverse limit of this system:
\begin{defn}[the Deligne-Ribet monoid $DR_K$]
The \emph{Deligne-Ribet monoid} $DR_K$ is the profinite monoid defined as the following inverse limit of the above inverse system of the ray class monoids $DR_\idf$:
\begin{eqnarray}
 DR_K &:=& \lim_{\idf \in I_K} DR_\idf. 
\end{eqnarray}
\end{defn}
\begin{rem}
By definition, $DR_K$ is commutative; also, for each $\idf \in I_K$, we can identify $DR_\idf^{\times} = C_\idf$, where $DR_\idf^\times$ is the unit group of $DR_\idf$ and $C_\idf$ is the strict ray class group with the conductor $\idf \cdot (\infty)$ (cf.\ 2.6, \cite{Deligne_Ribet}). Taking inverse limit, we then have an isomorphism $DR_K^\times  \simeq \lim_\idf C_\idf \simeq G_K^{ab}$, where the second isomorphism is the class-field-theory isomorphism. 
\end{rem}

The above-mentioned map $G_K \times I_K \rightarrow DR_K$ is then given, on the second factor, as the canonical map $I_K \rightarrow DR_K$; and on the first factor, as the composition $G_K \twoheadrightarrow G_K^{ab} \simeq \lim_\idf C^\idf \simeq DR_K^{\times} \subseteq DR_K$. With this, we restate here the result of Borger and de Smit \cite{Borger_Smit11}: 

\begin{thm}[Theorem 1.2 \cite{Borger_Smit11}]
A finite etale $\Lambda$-ring $X$ over $K$ has an integral model if and only if the $(G_K \times I_K)$-action on the finite set $\F_K (X)$ factors through the map $G_K \times I_K \rightarrow DR_K$ given above; in other words, this means that the fundamental monoid $\pi_1(\C_K^{op}, \F_K)$ is isomorphic to $DR_K$ and we have the following equivalence of categories:
\begin{eqnarray*}
  \C_K^{op} &\xrightarrow{\simeq}& \cl DR_K \\
  X &\longmapsto& \F_K(X)
\end{eqnarray*}
\end{thm}

\begin{rem}[every component of $X \in \C_K$ is abelian]
Recall that, since the underlying $K$-algebra of each object $X \in \C_K$ is finite etale over $K$, we have an isomorphism $X \simeq L_1 \times \cdots \times L_n$ for some finite extensions $L_i/K$ of number fields. As described above, moreover, since the action of $G_K$ onto $\F_K (X)$ factors through the abelianization $G_K \twoheadrightarrow G_K^{ab}$, it follows that every component $L_i$ of $X$ must be abelian over $K$. (See \S 3 \cite{Borger_Smit11} for a reason of this fact; this is not true for finite etale $\Lambda$-rings \emph{without} integral models.) We shall use this fact throughout this paper. 
\end{rem}

\begin{rem}[the case of $K = \ratf$]
In \cite{Borger_Smit08}, preceding \cite{Borger_Smit11}, the authors studied the case when $K$ is the rational number field $\ratf$. In this case, it is shown that the Deligne-Ribet monoid $DR_\ratf$ is isomorphic to the multiplicative monoid $\widehat{\integer}$ of profinite integers. In particular, the monoid $I_\ratf$ is identified with the multiplicative monoid $\nat$ of positive integers; and we can see that $DR_N$ for $N \in \nat$ is isomorphic to the multiplicative monoid $\integer/N \integer$. 
\end{rem}

\begin{rem}[galois objects]
\label{galois object of C_K}
The major subject of \S \ref{s3} is to classify the galois objects of $\C_K$ in terms of integral Witt vectors. For this reason, we recall here a few basic facts about galois objects of the semi-galois category $\C_K$; nevertheless, instead of recalling the fully general facts on galois objects (cf.\ \S 3.2 \cite{Uramoto17}), it is sufficient for our purpose to see some description of the galois objects of $\cl DR_K$. In this category $\cl DR_K$, the galois objects are precisely the \emph{rooted $DR_K$-sets}, i.e.\ those $DR_K$-sets $S \in \cl DR_K$ which have some $s_0 \in S$ such that $S = s_0 \cdot DR_K$; in this case, $s_0 \in S$ is called a \emph{root} of $S$. (In general, rooted objects are not galois; but in the current situation, $DR_K$ is commutative; from this, it follows that rooted objects in $\cl DR_K$ are always galois.) In terms of $\Lambda$-rings $X \in \C_K$, this means that $X \in \C_K$ is galois if and only if there exists $s_0 \in \F_K(X) = \Hom_K(X, \bar{K})$ such that every $s \in \F_K(X)$ can be written as $s = \sigma \circ s_0 \circ \psi_\ida$ for some $\sigma \in G_K$ and $\ida \in I_K$. We shall use this fact for the study in \S \ref{s3}. 
\end{rem}

\subsection{The moduli spaces of lattices}
\label{s2.2}
The Deligne-Ribet monoid $DR_K$ was defined as the inverse limit $\lim_\idf DR_\idf$ of the ray class monoids $DR_\idf$ ($\idf \in I_K$); but as proved by Yalkinoglu \cite{Yalkinoglu}, this profinite monoid $DR_K$ has yet another aspect as the ``moduli space'' of \emph{$1$-dimensional $K$-lattices} up to some scaling \cite{Connes_Marcolli_Ramachandran,Laca_Larsen_Neshveyev}. In the case when $K$ is an imaginary quadratic field in particular, such $K$-lattices are naturally \emph{$2$-dimensional $\ratf$-lattices}; therefore, $DR_K$ then has a natural embedding to the space $Lat_\ratf^2$ of $2$-dimensional $\ratf$-lattices up to scaling, which is further proved isomorphic to a certain quotient space $\Gamma \backslash (M_2(\widehat{\integer}) \times \upper)$ \cite{Connes_Marcolli}. In this subsection, we recall these facts essentially from \cite{Connes_Marcolli,Connes_Marcolli_Ramachandran,Yalkinoglu,Laca_Larsen_Neshveyev} which will provide a key for relating Witt vectors to modular functions. 

Throughout this subsection, let us suppose that $K$ is an imaginary quadratic field, and also let $O_K = \integer \tau + \integer$ with $\tau \in \upper$, where $\upper$ denotes the upper half plane in $\comp$. 
For a number field $K$, we denote by $\Ad_K$ the ring of adeles of $K$; by $\Ad_{K,f}$ the ring of finite adeles; and also, by $\widehat{O}_K$ the ring of finite integral adeles. In general, for a ring $R$, we denote by $R^*$ the invertible elements. Finally, $[-]: \Ad_K^* \rightarrow G_K^{ab}$ denotes the Artin map. 

We start with relating $DR_K$ with \emph{$1$-dimensional $K$-lattices}, which is defined as follows: 

\begin{defn}[$K$-lattices; \cite{Connes_Marcolli_Ramachandran}]
A ($1$-dimensional) \emph{$K$-lattice} is a pair $(\Lambda, \phi)$ of a finitely generated $O_K$-submodule $\Lambda$ of $\comp$ such that $K \otimes_{O_K} \Lambda \simeq K$ and a $O_K$-linear homomorphism $\phi: K/O_K \rightarrow K \Lambda / \Lambda$. 
\end{defn}

\begin{ex}[cf.\ Lemma 2.4 \cite{Connes_Marcolli_Ramachandran}]
Fractional ideals of $K$ are ``prototypical'' examples of $K$-lattices in the sense that, up to scaling by some $\lambda \in \comp^*$, every $K$-lattice $\Lambda \subseteq \comp$ becomes a fractional ideal as $\lambda \Lambda \subseteq K$. Such a scaling factor $\lambda \in \comp^*$ is unique modulo $K^*$.
\end{ex}


\begin{prop}[cf.\ Proposition 2.6 \cite{Connes_Marcolli_Ramachandran}]
There are bijective correspondences (1) between $\widehat{O}_K \times_{\widehat{O}_K^*} (\Ad_K^* / K^*)$ and the set of $K$-lattices, and (2) between $\widehat{O}_K \times_{\widehat{O}_K^*} (\Ad_{K,f}^* / K^*)$ and the set of $K$-lattices up to scaling. 
\end{prop}
\begin{proof}
For this result, the reader is referred to \cite{Connes_Marcolli_Ramachandran}; we just describe the correspondence. For $[\rho, t] \in \widehat{O}_K \times_{\widehat{O}_K^*} (\Ad_K^* / K^*)$, the correspondence is given by $[\rho, t] \mapsto (\Lambda_t, \phi_{(\rho, t)})$, where $\Lambda_t := t_\infty^{-1} (t_f \widehat{O}_K \cap K)$ with $t_f, t_\infty$ denoting the non-archimedian and archimedian components of $t = (t_f, t_\infty) \in \Ad_K^*$ and $\phi_{(\rho, t)}: K/O_K \rightarrow K \Lambda_t / \Lambda_t$ is given by the following composition (the upper row):
\begin{eqnarray}
\xymatrix{
 K/O_K \ar[r] \ar[d]_{\simeq} & K/O_K \ar[r]  \ar[d]_{\simeq} & K \Lambda_{t_f} / \Lambda_{t_f} \ar[r]^{t_\infty^{-1}} \ar[d]_{\simeq}  & K \Lambda_t / \Lambda_t  \\
\Ad_{K,f}/\widehat{O}_K \ar[r]_\rho & \Ad_{K,f} / \widehat{O}_K \ar[r]_{t_f}  & \Ad_{K,f} / t_f \widehat{O}_K  & }
\end{eqnarray}
where $\rho, t_f$ and $t_\infty^{-1}$ denote the maps given by the straightforward multiplications. Based on this, the second correspondence is given just by forgetting the scaling factor $t_\infty$. 
\end{proof}

\begin{prop}[$DR_K$ as moduli space]
There is a bijective correspondence between $DR_K$ and the set of $K$-lattices up to scaling. 
\end{prop}
\begin{proof}
This is given as a combination of the above results of \cite{Connes_Marcolli_Ramachandran} and that of \cite{Yalkinoglu} together with the class field theory isomorphism $[-]: \Ad_{K,f}^* / K^* \xrightarrow{\simeq} G_K^{ab}$. That is, by Proposition 8.2 \cite{Yalkinoglu}, we have an isomorphism $DR_K \simeq \widehat{O}_K \times_{\widehat{O}_K^*} G_K^{ab}$; then, by combining with $G_K^{ab} \simeq \Ad_{K,f}^* / K^*$, this induces an isomorphism $DR_K \simeq \widehat{O}_K \times_{\widehat{O}_K^*} (\Ad_{K,f}^*/K^*)$. As proved above, the latter corresponds to the set of $K$-lattices up to scaling, hence the claim. 
\end{proof}

In what follows, we shall denote by $Lat_K^1$ the set of $K$-lattices up to scaling; for each $K$-lattice $(\Lambda,\phi)$, we shall denote by $[\Lambda,\phi] \in Lat_K^1$ the equivalence class (up to scaling) of $(\Lambda, \phi)$. So far, we proved the following isomorphisms:
\begin{equation}
 DR_K \simeq \widehat{O}_K \times_{\widehat{O}_K^*} (\Ad_{K,f}^* / K^*) \simeq Lat_K^1. 
\end{equation}
Moreover, since $K$ is now imaginary quadratic, hence, of degree $2$ over $\ratf$, we can then think of $K$-lattices as \emph{$2$-dimensional $\ratf$-lattices} in the following sense:

\begin{defn}[$2$-dimensional $\ratf$-lattice]
A ($2$-dimensional) \emph{$\ratf$-lattice} is a pair $(\Lambda, \phi)$ of a lattice $\Lambda$ in $\comp$ and a homomorphism $\phi: \ratf^2 / \integer^2 \rightarrow \ratf \Lambda / \Lambda$. 
\end{defn}

Recall that we now have $O_K = \integer \tau + \integer$ for $\tau \in \upper$, hence, $K = \ratf \tau + \ratf$; this choice of the basis $(\tau, 1)^t$ of $K$ over $\ratf$ gives an isomorphism $K \simeq \ratf^2$, with $O_K \simeq \integer^2$. Also, note that for a $K$-lattice $\Lambda \subseteq \comp$, we have $K \Lambda = \ratf \Lambda \subseteq \comp$. By these identifications, each $K$-lattice $(\Lambda, \phi)$ naturally defines a ($2$-dimensional) $\ratf$-lattice $(\Lambda, \phi')$, where $\phi': \ratf^2/\integer^2 \rightarrow \ratf \Lambda/ \Lambda$ is given by:
\begin{equation}
\phi': \ratf^2/\integer^2 \simeq K / O_K \xrightarrow{\phi} K \Lambda/\Lambda = \ratf \Lambda/\Lambda.
\end{equation}
In this sense, we identify $K$-lattices $(\Lambda,\phi)$ with ($2$-dimensional) $\ratf$-lattices; and denote by the same symbol $(\Lambda,\phi)$. 

In general, as in the case of $K$-lattices, $2$-dimensional $\ratf$-lattices can be classified with a certain ``moduli space''; the following proposition essentially due to \cite{Connes_Marcolli} gives a description of the space: (In what follows, we simply denote $\Gamma := SL_2 (\integer)$.)
\begin{prop}[\cite{Connes_Marcolli}]
There is a bijective correspondence between $\Gamma \backslash (M_2(\widehat{\integer}) \times GL_2^+(\realf))$ and the set of $\ratf$-lattices. By this correspondence, we have a bijective correspondence between $\Gamma \backslash (M_2(\widehat{\integer}) \times \upper)$ and the set of $\ratf$-lattices up to scaling by $\comp^*$ as well. 
\end{prop}
\begin{proof}
The proof is essentially due to \cite{Connes_Marcolli}; but we need modify the constructions there so that our constructions below get compatible with the conventions in \cite{Shimura}. With this slight modification, the proof in \S \ref{s4} works well. To be more specific, for each $(m, \alpha) \in M_2(\widehat{\integer}) \times GL_2^+(\realf)$, the corresponding $\ratf$-lattice is defined by $(m, \alpha) \mapsto (\Lambda_\alpha, m \alpha)$, where $\Lambda_\alpha := \integer^2 \cdot \alpha \cdot (i, 1)^t$ regarding $\integer^2$ as consisting of row vectors, and $m \alpha: \ratf^2 /\integer^2 \rightarrow \ratf \Lambda_\alpha / \Lambda_\alpha$ is given by $a \mapsto a \cdot m \alpha \cdot (i, 1)^t$ for $a = (a_1, a_2) \in \ratf^2/\integer^2$. (Here $M_2(\widehat{\integer})$ acts on $\ratf^2/\integer^2$ as $\ratf^2/\integer^2 \simeq \Ad_{\ratf,f}^2/\widehat{\integer}^2$.) The group $\Gamma = SL_2(\integer)$ acts on $M_2(\widehat{\integer}) \times GL_2^+(\realf)$ by $(\gamma, (m,\alpha)) \mapsto (m \gamma^{-1}, \gamma \alpha)$. With this definition, the above correspondence $(m,\alpha) \mapsto (\Lambda_\alpha, m\alpha)$ induces a bijection from $\Gamma \backslash (M_2(\widehat{\integer}) \times GL_2^+(\realf))$ to the set of $\ratf$-lattices. 

Concerning the second claim, note that the action of scaling $\lambda = s + it \in \comp^*$ on $\Lambda_\alpha$ corresponds to the action of the following matrix on the right of $\alpha \in GL_2^+(\realf)$:
\begin{eqnarray}
 \lambda &=& \left ( \begin{array}{cc} s & -t \\ t & s \end{array} \right).
\end{eqnarray}
Also the quotient $GL_2^+(\realf)/ \comp^*$ by this action of $\comp^*$ can be identified with the upper half plane $\upper$ via the following correspondence $GL_2^+ (\realf) \rightarrow \upper$:
\begin{eqnarray} 
 \alpha = \left ( \begin{array}{cc} a & b \\ c & d \end{array} \right) &\longmapsto& \alpha(i) := \frac{ai + b}{ci + d}. 
\end{eqnarray}
By this identification, we eventually obtain a bijection from $\Gamma \backslash (M_2(\widehat{\integer}) \times \upper)$ to the set of $\ratf$-lattices up to scaling. 
\end{proof}

\begin{rem}[basis of $\Lambda_\alpha$]
Suppose that $\alpha \in GL_2^+(\realf)$ is given as follows:
\begin{eqnarray}
 \alpha &=& \left ( \begin{array}{cc} a & b \\ c & d \end{array} \right).
\end{eqnarray}
Then the basis of $\Lambda_\alpha$ is given by $(ai + b, ci + d)$; that is, $\Lambda_\alpha = \integer (ai+ b) + \integer (ci + d)$. Also, note that $\tau_\alpha := \alpha (i) = (ai + b) / (ci + d)$ corresponds to this lattice $\Lambda_\alpha$ up to scaling. 
\end{rem}

Let $Lat_\ratf^2$ denote the set of $2$-dimensional $\ratf$-lattices up to scaling. In the above lemma, we have constructed the isomorphism $Lat_\ratf^2 \simeq \Gamma \backslash (M_2(\widehat{\integer}) \times \upper)$; and we also have a natural embedding of sets $Lat_K^1 \hookrightarrow Lat_\ratf^2$ in the sense mentioned above. Consequently, we have constructed the following sequence of maps:
\begin{equation}
  DR_K \simeq \widehat{O}_K \times_{\widehat{O}_K^*} (\Ad_{K,f}^* / K^*) \simeq Lat_K^1 \hookrightarrow Lat_\ratf^2 \simeq \Gamma \backslash (M_2(\widehat{\integer}) \times \upper). 
\end{equation}
We will use this sequence of maps in \S \ref{s4} to relate Witt vectors with modular functions.

\section{Galois objects of $\C_K$}
\label{s3}
In this section, before proceeding to our major subject in \S \ref{s4}, we develop some general facts about the $\Lambda$-ring $K \otimes \intwitt$ of algebraic Witt vectors. The first subsection (\S \ref{s3.1}) gives a classification of galois objects of $\C_K$, from which we deduce the universality of $K \otimes \intwitt$ with respect to the $\Lambda$-rings in $\C_K$; the second subsection (\S \ref{s3.2}) studies the structure of the galois objects of $\C_K$, where we particularly determine the \emph{state complexity} of integral Witt vectors; the last subsection (\S \ref{s3.3}) gives a presentation of $\intwitt$ and $K \otimes \intwitt$ when $K$ is the rational number field $\ratf$, which are proved isomorphic to the group-rings $\integer[\ratf/\integer]$ and $\ratf[\ratf/\integer]$ respectively. After developing these basic results, the next section (\S \ref{s4}) then proceeds to our major subject, where we study these rings when $K$ is an imaginary quadratic field. 

\subsection{Classification of galois objects}
\label{s3.1}
\noindent
The subject of this subsection is to classify the galois objects in $\C_K$ in terms of integral Witt vectors; in this section, let $E_K$ denote the direct limit of the galois objects of $\C_K$ (see \S 3.1 \cite{Uramoto17} for the \emph{inverse} system of galois objects of $\C_K^{op}$). We prove that $E_K$ is naturally isomorphic to $K \otimes \intwitt$; this gives a characterization of our $\Lambda$-ring $K \otimes \intwitt$ of algebraic Witt vectors, which gives a basis of our study in \S \ref{s4}. 

To this end we need the following construction:

\begin{defn}[$\Lambda$-ring $X_\Xi$]
 Let $\Xi \subseteq \intwitt$ be a finite subset. Then the \emph{associated $\Lambda$-ring} $X_\Xi$ is defined as follows:
\begin{eqnarray}
   X_\Xi  &:=& K \otimes O_K [I_K \Xi],
\end{eqnarray}
where $I_K \Xi$ is the orbit of $\Xi$ under the action of $\psi_\ida$s ($\ida \in I_K$). 
\end{defn}

We show that these $\Lambda$-rings are always galois objects of $\C_K$; and conversely, every galois object of $\C_K$ is of this form up to isomorphism:

\begin{prop}
\label{X_Xi is galois}
 For any finite set $\Xi \subseteq \intwitt$, the associated $\Lambda$-ring $X_\Xi$ is a galois object with its root given by $X_\Xi \ni \eta \mapsto \eta_1 \in \bar{K}$. 
\end{prop}
\begin{proof}
 The proof of the first claim that $X_\Xi$ has an integral model (i.e.\ is an object of $\C_K$) follows similarly to that of Proposition 4, \cite{Uramoto18}. We prove that $X_\Xi$ is a galois object with the homomorphism $s_1: X_\Xi \ni \xi \mapsto \xi_1 \in \bar{K}$ being its root, where $1=O_K$ is the trivial  ideal of $O_K$. To be more specific, by definition of galois objects, it suffices to see that every homomorphism $s: X_\Xi \rightarrow \bar{K}$ is given by $s = \sigma \circ s_1 \circ \psi_\ida$ for some $\ida \in I_K$ and $\sigma \in G_K$ (cf.\ Remark \ref{galois object of C_K}, \S \ref{s2.1}). (In this proof, let us denote $X:=X_\Xi$ for simplicity.)
 
To see this, let us first consider a monoid congruence $\equiv_\Xi$ on $I_K$ defined by $\ida \equiv_\Xi \idb$ if and only if $\psi_\ida \xi = \psi_\idb \xi$ for every $\xi \in \Xi$, which is of finite index since the orbit of $\Xi$ under the action of $\psi_\ida$'s ($\ida \in I_K$) is finite (cf.\ Theorem 3.4 \cite{Uramoto18}). Let $I_K = J_1 \sqcup \cdots \sqcup J_N$ be the $\equiv_\Xi$-class decomposition; and choose their representatives, say $\ida_i \in J_i$ for $i = 1, \cdots, N$. Also, since $X$ is finite over $K$ with each component being abelian, we can take a finite abelian subfield $\bar{K}/L/K$ so that the image $s(X)$ for every $s: X \rightarrow \bar{K}$ is contained in $L$. Then let us define a homomorphism $t: X \rightarrow L \times \cdots \times L$ by $t(\eta) := (\eta_{\ida_1}, \cdots, \eta_{\ida_N})$, which is injective. In fact, note first that if $\ida \equiv_\Xi \idb$ then $\eta_\ida = \eta_\idb$ for every $\eta \in X_\Xi$ by the fact that $X_\Xi$ is generated by the orbit $I_K \Xi$ and by definition of $\equiv_\Xi$. Therefore, the values of $\eta_\ida$ are determined by those of $\eta_{\ida_i}$ for $i=1,\cdots, N$; hence $t: X \rightarrow L \times \cdots \times L$ is injective. This means that the $G_K$-set $\Hom_K(X, \bar{K})$ is a quotient of $\bigsqcup \Hom_K(L, \bar{K})$ induced from $t: X \rightarrow L \times \cdots \times L$; and by construction, this proves the claim that every $s \in \Hom_K(X,\bar{K})$ is given by $s = \sigma \circ s_1 \circ \psi_{\ida_i}$ for some $i$. 
\end{proof}

\begin{prop}
\label{galois is X_Xi}
 Every galois object $X \in \C_K$ is isomorphic to $X_\Xi$ for some finite subset $\Xi \subseteq \intwitt$. 
\end{prop}
\begin{proof}
 Let $s: X \rightarrow \bar{K}$ be a root of $X$ and $A \leq X$ be its integral model. For each $\xi \in X$ we define $\xi^s = (\xi_\ida^s) \in \bar{K}^{I_K}$ by $\xi^s_\ida := s (\psi_\ida \xi)$. It is not difficult to see that $\xi^s \in W_{O_K}^a(O_{\bar{K}})$ for every $\xi \in A$ because $A$ is an integral model. Also, since $X = K \otimes A$ and $A$ is finite over $O_K$, there exist some $\xi_1, \cdots, \xi_n \in A$ that generate $X$ over $K$. Let $\Xi := \{\xi_1^s, \cdots, \xi_n^s\} \subseteq W_{O_K}^a(O_{\bar{K}})$. We prove that $X$ is isomorphic to $X_\Xi$. Notice that the assignment $\xi \mapsto \xi^s $ for $\xi \in X$ defines a $\Lambda$-ring homomorphism $X \rightarrow X_\Xi$, which is clearly surjective. To show its injectivity, note that since $s$ is a root of the galois object $X$, every homomorphism $s': X \rightarrow \bar{K}$ is a composition $s'= s \circ f$ for some $f \in \End(X)$. Also, by the fact that $\End(X)$ is generated by $\psi_\idp$'s, this implies that the values $\xi_\ida^s = s (\psi_\ida \xi)$ determines $\xi \in X$; hence, $\xi \mapsto \xi^s$ is indeed injective. 
\end{proof}

Therefore, the galois objects in $\C_K$ are precisely of the form $X_\Xi$ for some finite $\Xi \subseteq \intwitt$. Consequently, we obtain the following presentation of the $K$-algebra $E_K$ in terms of integral Witt vectors: 

\begin{thm}
\label{E_K}
 We have a canonical isomorphism of $\Lambda$-rings: 
 \begin{eqnarray}
   E_K &\simeq& K \otimes \intwitt. 
 \end{eqnarray}
In particular, $\intwitt$ is isomorphic to the direct limit of the maximal integral models of galois objects of $\C_K$. 
\end{thm}
\begin{proof}
 This isomorphism $E_K \rightarrow K \otimes \intwitt$ is given as (the direct limit of) the isomorphisms $X \rightarrow X_\Xi$ constructed in the above proposition for galois objects $X$. 
\end{proof}

\begin{rem}
In what follows we identify $E_K$ with $K \otimes \intwitt$ in the sense of this isomorphism, and often call $E_K$ as the $K$-algebra of algebraic Witt vectors. 
\end{rem}


\begin{cor}
\label{characterization of E_K}
 The $K$-algebra $E_K = K \otimes \intwitt$ is isomorphic to the $K$-algebra of all locally constant $K^{ab}$-valued $G_K$-equivariant functions on $DR_K$, symbolically, $E_K = \Hom_{G_K}(DR_K, \bar{K})$. 
\end{cor}
\begin{proof}
 This is a direct consequence of Theorem \ref{E_K} above, and see also Theorem 10.1 \cite{Yalkinoglu} due to Neshveyev. To be more precise, Theorem 10.1 \cite{Yalkinoglu} claims that $E_K$ is isomorphic to the $K$-algebra of such functions. Since we proved that $E_K$ is also isomorphic to $K \otimes \intwitt$, the composition of these isomorphisms shows the claim of this corollary. 
\end{proof}

\begin{rem}[cyclic Witt vector]
\label{cyclic witt vector}
We are concerned with whether every galois object $X \in \C_K$ is actually isomorphic to $X_\xi$ for some $\xi \in \intwitt$, that is, whether we can take $\Xi \subseteq \intwitt$ in Proposition \ref{galois is X_Xi} as a singleton $\Xi = \{\xi\}$. This is true for e.g.\ the (cofinal) galois objects of the form $\ratf[x]/(x^n -1)$ of $\C_\ratf$; but we do not know whether every galois object $X$ has such a $\xi$ in general, which we shall call a \emph{cyclic Witt vector} for $X$ (a la, cyclic vectors for differential modules known in differential galois theory). 
\end{rem}

\begin{rem}[galois correspondence]
Although we do not give a proof here, it would be meaningful to mention a certain galois correspondence that naturally extends the usual galois correspondence of galois theory for number fields. To be specific, our galois correspondence is the one between the following objects:
\begin{enumerate}
 \item $\Lambda$-subalgebras of $E_K$; 
 \item profinite quotients of $DR_K$;
 \item semi-galois full subcategories of $\C_K$. 
\end{enumerate}
This follows from the presentation of $E_K = \Hom_{G_K}(DR_K, \bar{K})$ and $DR_K = \Hom_K(E_K, \bar{K})$ as well as the duality between $DR_K$ and $\C_K$. 
\end{rem}

\subsection{The structure of galois objects}
\label{s3.2}
\noindent 
As we proved in the above subsection, the galois objects of $\C_K$ are precisely those of the form $X_\Xi$ for some finite $\Xi \subseteq \intwitt$. In this subsection, we then study the structure of the galois objects $X_\Xi$ and represent $X_\Xi$ in terms of $\xi \in \Xi$. Starting from some general facts about galois objects in $\C_K$, we describe the components $L_i$ of $X_\Xi \simeq L_1 \times \cdots \times L_r$ and also determine the \emph{state complexity} of integral Witt vectors $\xi \in \intwitt$, which is a natural analogue of Bridy's result \cite{Bridy} on formal power series $\xi \in \fld_q [[t]]$ algebraic over $\fld_q[t]$. 

To this end, we need to prepare some general lemmas:

\begin{lem}
\label{lem1}
\label{automorphism of psi}
 Let $X \in \C_K$. If $\psi_\idp$ is an automorphism on $X$, then the action of $\psi_\idp$ on $\F_K(X)$ is equal to that of some $\sigma \in G_K$. 
\end{lem}
\begin{proof}
 The proof is done by completion and Theorem1.1, \cite{Borger_Smit18}. Let $A \leq X$ be the maximal integral model of $X$ and $K_\idp$ be the completion of $K$ at $\idp$; also let $X_\idp = K_\idp \otimes X$ and $A_\idp = O_{K_\idp} \otimes A$. Then $X_\idp$ together with $\psi_\idp$ is a $\Lambda_\idp$-ring with integral model $A_\idp$. Since $\psi_\idp$ is an automorphism on $X_\idp$, the set $S_\mathrm{unr} := \bigcap_{n=0}^\infty S \idp^n$, where $S := \F_{K_\idp} (X_\idp)$, is equal to $S$ itself. This implies that, by Theorem 1.1, \cite{Borger_Smit18}, the action of $\idp$ on $S_\mathrm{unr} = S$ is equal to the Frobenius $\sigma_\idp \in G_{K_\idp} / I_\idp$ (where $I_\idp \leq G_{K_\idp}$ denotes the inertia subgroup). 
\end{proof}

\begin{lem}
\label{lem2}
Let $X \in \C_K$ be a galois object. If $f \in \End(X)$ is an automorphism of the $\Lambda$-ring $X$, then the action of $f$ on $\F_K(X)$ by pullback is equal to the action of some $\sigma \in G_K$. 
\end{lem}
\begin{proof}
Let $f \in \End(X)$ be an automorphism. Since $X$ is galois and $I_K$ is dense in $DR_K$, the action of $f$ on $\F_K (X)$ by pullback is equal to that of $\psi_\ida$ for some $\ida \in I_K$. (To see this, recall the definition of galois objects, Definition 12, \S 3.2 \cite{Uramoto17}, and the fact that the galois objects of $\cl DR_K$ are those $DR_K$-sets which are of the form of finite quotients $DR_K \twoheadrightarrow H$, cf.\ Lemma 17, \S 4.1 \cite{Uramoto17}.) Now since $f$ is an automorphism, so is the action of $\psi_\ida$. Hence, if $\ida = \idp_1 \cdots \idp_n$, the actions of $\psi_{\idp_i}$'s are all automorphisms as well. By the above lemma, the actions of $\psi_{\idp_i}$'s come from some $\sigma_i \in G_K$; hence, so is the action of $\psi_\ida = \psi_{\idp_1} \cdots \psi_{\idp_n}$, which is equal to that of $f$. 
\end{proof}

To proceed further, let us recall the following notion from semigroup theory:

\begin{defn}[$\J$-equivalence]
 Let $M$ be a commutative monoid. For two elements $s ,s' \in M$, we denote as $s \leq_\J s'$ if we have an inclusion of the (two-sided) ideals $sM \subseteq s'M$, or in other words, there exists $m \in M$ such that $s = s'm$. Furthermore, we denote as $s \sim_\J s'$ and say that $s$ and $s'$ are \emph{$\J$-equivalent} if $s \leq_\J s'$ and $s' \leq_\J s$. The set of $\J$-equivalent classes is denoted as $M/\J$. 
\end{defn}

\begin{lem}
\label{lem3}
Let $X \in \C_K$ be galois with $s_1: X \rightarrow \bar{K} \in \F_K (X)$ its root. Then, $f, f' \in \End(X)$ are $\J$-equivalent in the (commutative) monoid $\End(X)$ if and only if the corresponding $s_1 \circ f, s_1 \circ f' \in \F_K(X)$ belong to the same $G_K$-component.
\end{lem}
\begin{proof}
For the reason mentioned in the above lemma, we may put $f = \psi_\ida$  and $f' = \psi_{\ida'}$ for some $\ida, \ida' \in I_K$. First suppose that $s_\ida:= s_1 \circ \psi_\ida$ and $s_{\ida'} := s_1 \circ \psi_{\ida'}$ belong to the same $G_K$-component in $\F_K (X)$; that is, $s_\ida = \sigma \circ s_{\ida'}$ for some $\sigma \in G_K$. By the density of $I_K$ in $DR_K$, the action of $\sigma$ is equal to the action of $\psi_\idb$ for some $\idb \in I_K$, which means that $s_\ida = s_{\ida' \idb}$, hence we have $\psi_\ida \leq_\J \psi_\ida'$ in $\End(X)$. (This is because $(X, s_1)$ is now galois, thus the assignment $\End(X) \ni f \mapsto s_1 \circ f \in \F_K(X)$ is injective; cf.\ Proposition 4, \S 3.2 \cite{Uramoto17}.) The converse inequality $\psi_\ida' \leq_\J \psi_\ida$ is similar, thus, $\psi_\ida$ and $\psi_{\ida'}$ are $\J$-equivalent. Second suppose that $\psi_\ida$ and $\psi_{\ida'}$ are $\J$-equivalent in $\End(X)$, whence $\psi_\ida = \psi_{\ida' \idb}$ and $\psi_{\ida'} = \psi_{\ida \idb'}$ for some $\idb, \idb' \in I_K$. Dually, i.e.\  in terms of the $DR_K$-set $\F_K(X)$, this means that $s_\ida$ and $s_{\ida'} \in \F_K(X)$ are mutually accessible by the action of $DR_K$, that is, $s_\ida \cdot DR_K = s_{\ida'} \cdot DR_K$. We put $S: = s_\ida \cdot DR_K$ = $s_{\ida'} \cdot DR_K$. Then, note that this $S$ forms a rooted $DR_K$-set, hence, galois in $\cl DR_K$ (cf.\ Remark \ref{galois object of C_K}, \S \ref{s2.1}); and also that, by $\psi_\ida = \psi_{\ida' \idb}$ and $\psi_{\ida'} = \psi_{\ida \idb'}$, one sees that the actions of $\psi_\idb, \psi_{\idb'}$ on $S$ are (mutually inverse) automorphisms. Thus, by Lemma \ref{lem2} applied to (the $\Lambda$-ring in $\C_K$ dual to) this galois object $S \in \cl DR_K$, it follows that the actions of $\psi_\idb, \psi_{\idb'}$ on $S$ are equal to those of some $\sigma, \sigma' \in G_K$, which implies that $s_\ida = s_{\ida' \idb}$ and $s_{\ida'} = s_{\ida \idb'}$ belong to the same $G_K$-component. This completes the proof. 
\end{proof}

Now we study the structure of the galois objects $X_\xi$ for $\xi \in \intwitt$. (For simplicity, we shall study the structure of $X_\Xi$ only for singleton $\Xi = \{\xi\}$, but a similar result holds in general.) Since $X_\xi$ is finite etale over $K$, we have an isomorphism $X_\xi \simeq L_1 \times \cdots \times L_r$ for some finite extensions $L_i/K$. However, we do not yet quite know what and how many components $L_i$ each $X_\xi$ has. In the following, we first discuss this problem. In this relation, let us denote by $M_\xi$ the quotient monoid $I_K/ \equiv_\xi$, where $\ida \equiv_\xi \idb$ if and only if $\psi_\ida \xi = \psi_\idb \xi$; by Theorem 3.2, \cite{Uramoto18}, $M_\xi$ is a finite monoid.

\begin{prop}
\label{presentation of X_xi}
For any $\xi \in \intwitt$ we have the following isomorphism:
\begin{eqnarray}
 X_\xi &\simeq& \prod_{[\ida] \in M_\xi / \J} K(\xi_{\ida \idb} ; \idb \in I_K).
\end{eqnarray}
In particular, the number of components is equal to $\# (M_\xi/ \J)$. 
\end{prop}
\begin{proof}
 For short let us put the right hand side as $Y_\xi :=  \prod_{[\ida] \in M_\xi / \J} K(\xi_{\ida \idb} ; \idb \in I_K)$; we construct the target isomorphism $f: X_\xi \rightarrow Y_\xi$. Let $M_\xi / \J = \{[\ida_1], \cdots, [\ida_r]\}$, and for each $[\ida_i] \in M_\xi / \J$, denote its representative as $\ida_i \in I_K$. Then we can define $f: X_\xi \rightarrow Y_\xi$ by $f (\eta) := (\eta_{\ida_i})_{i=1}^r$. To prove the lemma, note first that this $f$ is injective: If $f(\eta) = f(\eta')$ then $\eta_{\ida_i} = \eta'_{\ida_i}$ for every $i = 1, \cdots, r$. But for each $\ida \in I_K$ we have $[\ida] = [\ida_i]$ for some $i$; this means that, for every $\zeta \in X_\xi$, we have $\zeta_\ida = \zeta_{\ida_i}^\sigma$ for some $\sigma \in G_K$. (In fact $[\ida] = [\ida_i]$ implies $\ida = \ida_i \idb$ and $\ida_i = \ida \idb'$ for some $\idb, \idb' \in I_K$ in $M_\xi$, whence $\psi_\ida = \psi_{\ida_i} \psi_\idb$ and $\psi_{\ida_i} = \psi_\ida \psi_{\idb'}$ on $X_\xi$; then apply Lemma \ref{lem3} to the root $s_1: X_\xi \ni \zeta \mapsto \zeta_1 \in \bar{K}$ where $1$ is the unit in $I_K$.) Thus $\eta_\ida = \eta_{\ida_i}^\sigma = {\eta'_{\ida_i}}^\sigma = \eta'_\ida$ for every $\ida \in I_K$, which implies $\eta = \eta'$ as requested. 
 
Finally we see that $f$ is surjective. First note that, since $X_\xi$ is finite etale over $K$, we have an isomorphism $X_\xi \simeq L_1 \times \cdots \times L_m$ for some finite extensions $L_i / K$; and each $L_i$ is obtained as the image of $X_\xi$ under some $s_i : X_\xi \rightarrow \bar{K} \in \F_K(X_\xi)$. As shown in Proposition \ref{X_Xi is galois}, each $s_i \in \F_K(X_\xi)$ is given as $s_i(\eta) = \eta_{\ida_i}^{\sigma_i}$ ($\forall \eta \in X_\xi$) for some $\ida_i \in I_K$ and $\sigma_i \in G_K$, namely, $s_i$ is $G_K$-equivalent to $s'_i \in \F_K(X_\xi)$ given by $s'_i (\eta) = \eta_{\ida_i}$. So, the image $L_i$ of $X_\xi$ under $s_i$ is isomorphic to $K(\eta_{\ida_i} ; \eta \in X_\xi)$, which is further isomorphic to $K(\xi_{\ida_i \idb} ; \idb \in I_K)$ because $X_\xi$ is generated by $\psi_\idb \xi$'s ($\idb \in I_K$) over $K$. By Lemma \ref{lem3}, the two maps $s'_i$ and $s'_j$ are $G_K$-equivalent if and only if $[\ida_i] = [\ida_j]$. 
This means that the $G_K$-equivalent classes of $s \in \F_K(X_\xi)$ are classified precisely by the set $M_\xi / \J = \{[\ida_1], \cdots, [\ida_r]\}$ with $s'_i: \eta \mapsto \eta_{\ida_i} \in \F_K(X_\xi)$ for each $[\ida_i] \in M_\xi / \J$ being representative. Therefore, one concludes that $m = r$ and $L_i \simeq K(\xi_{\ida_i \idb}; \idb \in I_K)$, which completes the proof. 
\end{proof}

Finally, we determine the \emph{state complexity} of integral Witt vectors $\xi \in \intwitt$ (cf.\ Definition \ref{state complexity}) as a natural analogue of the result of Bridy \cite{Bridy} for formal power series $\xi \in \fld_q[[t]]$, where using the Riemann-Roch theorem he gave a sharp estimate of the state complexity of a formal power series $\xi \in \fld_q[[t]]$ algebraic over $\fld_q[t]$ in terms of the dimension $\dim_{\fld_q(t)} X_\xi$ of the function field $X_\xi$ of the curve generated by $\xi$ over $\fld_q(t)$. Analogously, we now show that the state complexity of an integral Witt vector $\xi \in \intwitt$ is \emph{equal} to the dimension $\dim_K X_\xi$ of the $\Lambda$-ring $X_\xi$. In fact, this holds for general $\Xi \subseteq \intwitt$. 

To be precise, the \emph{state complexity} of a finite set $\Xi \subseteq \intwitt$ is defined as follows: (See \S 3 \cite{Uramoto18} for the concept of DFA's and \emph{DFAO's} (i.e.\ \emph{deterministic finite automata with output}), and how they generate integral Witt vectors.)

\begin{defn}[state complexity]
\label{state complexity}
 Let $\Xi \subseteq \intwitt$ be any finite set of integral Witt vectors. We say that a DFA $\dfa$ \emph{generates $\Xi$} if for every $\xi \in \Xi$ there exists an output function $\tau: S_\dfa \rightarrow O_{\bar{K}}$ such that $\dfa_\tau$ generates $\xi$. The \emph{state complexity of $\Xi$}, denoted by $c_\Xi$, is then defined as the minimum size $\min \# S_\dfa$ of the state set $S_\dfa$ of those DFA $\dfa$ which can generate $\Xi$. 
\end{defn}

\begin{prop}[estimate of state complexity]
\label{estimate of state complexity}
 For any finite $\Xi \subseteq \intwitt$, we have the following identity:
 \begin{eqnarray}
   c_\Xi &=& \dim_K X_\Xi. 
 \end{eqnarray}
 In particular, $c_\xi = \dim_K X_\xi$ for $\xi \in \intwitt$.
\end{prop}
\begin{proof}
 Firstly note that we have $\dim_K X_\Xi = \# \F_K (X_\Xi)$. Also it is easy to see that $\F_K(X_\Xi)$ forms (the state set of) a DFA over $P_K$ that generates $\Xi$; thus we have the one-side inequality:
\begin{eqnarray} 
    c_\Xi &\leq& \# \F_K(X_\Xi)\\
             & = & \dim_K X_\Xi. 
\end{eqnarray}
To prove the inverse inequality, note that by Proposition \ref{X_Xi is galois}, we know that $X_\Xi$ is a galois object; hence, $\# \F_K(X_\Xi) = \# \End(X_\Xi)$. Therefore, it suffices to prove:
\begin{eqnarray}
   c_\Xi &\geq& \# \End(X_\Xi). 
\end{eqnarray}
To this end, let $\dfa = (S, \delta, s_0)$ be a minimum DFA generating $\Xi$. Then, by the minimality, every $s \in S$ is represented as $s = s_0 \cdot u$ for some $u = \idp_1 \cdots \idp_n \in P_K^*$. Also, we can see that if $s_0 \cdot u = s_0 \cdot v$ for $u, v \in P_K^*$, then $\psi_{\ida_u} \xi = \psi_{\ida_v} \xi$ for every $\xi \in \Xi$, where $\ida_w = \idp_1 \cdots \idp_n \in I_K$ for $w = \idp_1 \cdots \idp_n \in P_K^*$. In fact, for each $\xi \in \Xi$ there exists an output function $\tau: S \rightarrow O_{\bar{K}}$ such that for every $w \in P_K^*$:
\begin{eqnarray}
  (\psi_{\ida_u} \xi)_{\ida_w} &=& \xi_{\ida_{uw}} \\
  &=& \tau ((s_0 \cdot u) \cdot w) \\
  &=& \tau ((s_0 \cdot v) \cdot w) \\
  &=& (\psi_{\ida_v} \xi)_{\ida_w},
\end{eqnarray}
which means that $\psi_{\ida_u} \xi = \psi_{\ida_v} \xi$. Since $X_\Xi$ is generated by $I_K \Xi$, this then implies that $\psi_{\ida_u} = \psi_{\ida_v}$ on $X_\Xi$; hence, the assignment $S \ni s = s_0 \cdot u \mapsto \psi_{\ida_u} \in \End(X_\Xi)$ is well-defined. Furthermore, by the fact that $I_K \hookrightarrow DR_K$ is dense, this assignment $S \rightarrow \End(X_\Xi)$ is in fact surjective, which proves the desired inequality $c_\Xi = \# S \geq \# \End(X_\Xi)$. 
\end{proof}

\begin{rem}
\label{geometric background}
Note that the identity $c_\xi = \dim_K X_\xi$ relates seemingly unrelated quantities, namely the state complexity $c_\xi$ and the dimension $\dim_K X_\xi$ of the $K$-algebra $X_\xi$. In fact, while the former is just the size of the orbit $I_K \xi$ under the action of the Frobenius $\psi_\idp$, the latter is the quantity relevant to the algebraic degree of its coefficients $\xi_\ida$'s. As we discuss below, this identity seems related to the difficulty of actually constructing integral Witt vectors; to our thought, this is because there must be some geometric reason for why the coefficients $\xi_\ida$'s of integral Witt vectors $\xi$ distribute as they do. Concerning this, it will be meaningful to remark that the \emph{$j$-invariants} $j(\ida)$ in the theory of complex multiplication actually constitute an example of integral Witt vector; this is proved from their reciprocity law (cf.\ \S \ref{s4.2}).
\end{rem}

\subsection{The structure of $W_\integer^a (\bar{\integer})$}
\label{s3.3}
As an immediate consequence of the above general development, applied to the case $K=\ratf$, this section concludes with a presentation of the ring $W_\integer^a (\bar{\integer})$, which asserts that $W_\integer^a (\bar{\integer})$ is isomorphic to the group-ring $\integer[\ratf/ \integer]$. After this presentation we proceed to the case when $K$ is an imaginary quadratic field. 

\begin{cor}
\label{integral}
 The ring $W_\integer^a (\bar{\integer})$ is isomorphic to the group-ring $\integer[\ratf/\integer]$. 
\end{cor}
\begin{proof}
We first see that $E_\ratf = \ratf \otimes W_\integer^a (\bar{\integer})$ is isomorphic to the group-ring $\ratf[\ratf/\integer]$. In fact, by the result of Borger and de Smit \cite{Borger_Smit08} combined with our result in \S \ref{s3.1}, we know that $E_\ratf$ is isomorphic to the $\ratf$-algebra given by the direct limit $\lim_n \ratf[x] / (x^n -1) = \bigcup_n \ratf[x]/(x^n -1)$; here this direct limit of $\ratf$-algebras $\ratf[x]/(x^n - 1)$ is given by  the following embeddings, for positive integers $n, m$:
\begin{eqnarray}
\ratf[x]/(x^n -1) &\rightarrow& \ratf[x]/(x^{nm} - 1) \\
x &\longmapsto& x^m.
\end{eqnarray}
The group-ring $\ratf[\ratf/\integer]$, on the other hand, is also isomorphic to this $\ratf$-algebra $\bigcup_n \ratf[x]/(x^n -1)$ by the correspondence:
\begin{eqnarray} 
 \ratf[\ratf/\integer] &\rightarrow& \bigcup_n \ratf[x]/(x^n - 1) \\
 1/N &\longmapsto& x \in \ratf[x]/ (x^N -1). 
\end{eqnarray}
Hence we have $E_\ratf \simeq \ratf[\ratf/\integer]$. Finally we see that, in this isomorphism, the subring $W_\integer^a (\bar{\integer}) \subseteq E_\ratf$ corresponds to $\integer[\ratf/\integer] \subseteq \ratf[\ratf/\integer]$. Indeed, as proved in Theorem 3.4 \cite{Borger_Smit08}, the maximal integral model of the $\Lambda$-ring $\ratf[x]/(x^n -1)$ is $\integer[x]/(x^n -1)$ for each $n$; and thus, the subring $\bigcup_n \integer[x]/(x^n -1) \subseteq \bigcup_n \ratf[x]/(x^n -1)$ corresponds to $W_\integer^a (\bar{\integer}) \subseteq E_\ratf$. On the other hand, we see that $\bigcup_n \integer[x]/(x^n -1)$ corresponds to $\integer[\ratf/\integer]$ under the above isomorphism $\ratf[\ratf/\integer] \simeq \bigcup_n \ratf[x]/(x^n -1)$. Consequently, we obtain the desired isomorphism $W_\integer^a(\bar{\integer}) \simeq \integer[\ratf/\integer]$. This completes the proof. 
\end{proof}

\begin{rem}
More explicitly, in the above isomorphism $W_\integer^a(\bar{\integer}) \simeq \integer[\ratf/\integer]$, each $\gamma \in \ratf/\integer \subseteq \integer[\ratf/\integer]$ corresponds to the following Witt vector $\zeta^{(\gamma)} \in W_\integer^a(\bar{\integer}) \subseteq (\ratf^{ab})^{\nat}$:
\begin{eqnarray}
 \zeta^{(\gamma)} &:=& (e^{2\pi i \gamma n})_{n \in \nat}. 
\end{eqnarray}
Therefore we can conclude that, in general, the integral Witt vectors $\xi \in W_\integer^a(\bar{\integer})$ are precisely the $\integer$-linear combinations of these $\zeta^{(\gamma)}, \gamma \in \ratf/\integer$ (while the algebraic Witt vectors $\xi \in E_\ratf$ are the $\ratf$-linear combinations of $\zeta^{(\gamma)}$'s). This provides a complete classification of the integral Witt vectors over $\integer$. 
\end{rem}

\section{The modularity theorem}
\label{s4}
In some sense, the above isomorphism $E_\ratf \simeq \ratf[\ratf/\integer]$ clarifies how the coefficients $\xi_n$ of an algebraic Witt vector $\xi$ over $\ratf$ are correlated; they are essentially correlated as special values $e^{2\pi i \gamma n}$ for fixed $\gamma \in \ratf/\integer$ of the function $e^z$. 

The purpose of this section is to prove an analogue of this result in the case where $K$ is an imaginary quadratic field. To be precise, this section proves that the $K$-algebra $E_K = K \otimes \intwitt$ for an imaginary quadratic field $K$ is isomorphic to (or actually coincides with) the $K$-algebra $M_K$ that consists of \emph{modular vectors}; technically speaking, modular vectors are defined by special values of certain deformations of modular functions (\S \ref{s4.2}). Our target theorem, which we call the \emph{modularity theorem}, exhibits that such vectors arising from deformations of modular functions always define algebraic Witt vectors, and conversely, every algebraic Witt vector arises in this way (\S \ref{s4.3}). After this, we then give a specific set of generators of $E_K = M_K$ arising from the \emph{Fricke functions} $f_a$ (cf.\ \S \ref{s4.4}).

\subsection{Modular vectors}
\label{s4.2}
Before developing general results, let us start with some motivating observation. As mentioned in Remark \ref{geometric background}, the \emph{j}-function, i.e.\ a prototypical example of modular function, naturally induces an example of integral Witt vector; indeed, as we see below, this fact follows from \emph{Hasse's reciprocity law} of the $j$-invariants $j(\ida)$ (cf.\ \cite{Shimura}). This basic observation gives us yet another way to look at integral Witt vectors and the $j$-function; and then, leads us to the general consideration of \emph{modular vectors} developed below. 

To be more specific, the $j$-function induces an integral Witt vector $\iota \in \intwitt$ by its special values in the following way: 
\begin{prop}
\label{witt vector from j function}
Define $\iota = (\iota_\ida) \in O_{\bar{K}}^{I_K}$ by $\iota_\ida := j(\ida^{-1})$ for each $\ida \in I_K$; and let $H_K$ denote the Hilbert class field of $K$. Then $\iota$ is an integral Witt vector with coefficients in $O_{H_K}$:
\begin{eqnarray}
 \iota = (\iota_\ida) &\in& W_{O_K}^a(O_{H_K}).
\end{eqnarray}
\end{prop}
\begin{proof}
To see this, recall first that every prime $\idp \in P_K$ is unramified in $H_K$; therefore $H_K$ in itself forms an object of $\C_K$ together with the Frobenius automorphisms $\sigma_\idp \in Gal(H_K / K)$ ($\idp \in P_K$) as the $\Lambda$-structure $\psi_\idp$; there, $O_{H_K}$ is the maximal integral model. 

With this in mind, take the natural embedding $s: H_K = K(j(1)) \hookrightarrow \bar{K} \subseteq \comp$; and for $\xi = j(1) \in O_{H_K}$, consider $\xi^s \in \bar{K}^{I_K}$ as in Proposition \ref{galois is X_Xi}. Then, as discussed there, we have $\xi^s \in W_{O_K}^a(O_{\bar{K}})$ because $j(1) \in O_{H_K}$. To be more specific, the component $\xi_\ida^s$ of $\xi^s \in \intwitt$ at each $\ida \in I_K$ is given by $s (\psi_\ida \xi)$; but by construction of $\xi = j(1) \in H_K$ and Hasse's reciprocity law (i.e.\ $j(\idb)^{\sigma_\idp} = j(\idp^{-1}\idb)$ for $\idp \in P_K$ and fractional ideals $\idb$ of $K$), this means that $\xi_\ida^s = j(\ida^{-1}) = \iota_\ida$ for any $\ida \in I_K$, hence the claim. 
\end{proof}

While algebraic (or integral) Witt vectors are defined in a purely algebraic way, this proposition indicates us an analytic way to construct a prototypical example of them. In fact we now see that this observation is a special case of a more general relationship between algebraic Witt vectors and modular functions; the following concept of \emph{Witt deformation family of modular functions} is a key to explain this general relationship. In what follows, let us denote by $\modular = \bigcup_N \modular_N$ the field of modular functions of finite level \emph{rational over} $\ratf^{ab}$, where $\modular_N$ is the field of modular functions of level $N$ \emph{rational over} $\ratf(e^{2 \pi i /N})$ in the sense of \S 6.2, pp.137 \cite{Shimura}, i.e.\ $\modular_N = \ratf (j, f_a \mid a \in N^{-1}\integer^2 / \integer^2)$ with $f_a$ \emph{Fricke functions} (cf.\ Example \ref{Fricke functions}). (Recall that the field of all modular functions of finite level is equal to $\comp \otimes_\ratf \modular$; cf.\ Proposition 6.1, \cite{Shimura}.) 

\begin{defn}[Witt deformation family of modular functions]
\label{Witt deformation family}
A \emph{Witt deformation family of modular functions} (or simply, a \emph{Witt deformation of modular functions}) is a continuous function\footnote{To be more precise, $f: M_2(\widehat{\integer}) \times \upper \rightarrow \comp$ is actually a partial function or a function to $\comp \cup \infty$ in general; but in this paper we shall denote $f$ as a function to $\comp$ just for simplicity.} $f: M_2(\widehat{\integer}) \times \upper \rightarrow \comp$ satisfying the following axioms:
\begin{enumerate}
 \item for each $m \in M_2(\widehat{\integer})$, the fiber $f_m := f(m, -): \upper \rightarrow \comp$ is a modular function in $\modular$;
 \item for each $m \in M_2(\widehat{\integer})$ and $u \in GL_2(\widehat{\integer})$, we have:
 \begin{eqnarray}
   f_{m u} &=& f_m^u; 
 \end{eqnarray}
 where $f_m \mapsto f_m^u$ denotes the action of $u \in GL_2(\widehat{\integer})$ onto the modular field $\modular$ (cf.\ \S 6.6 \cite{Shimura}); 
 \item the function $f: M_2(\widehat{\integer}) \times \upper \rightarrow \comp$ factors through the projection $p_N: M_2(\widehat{\integer}) \twoheadrightarrow M_2(\integer / N \integer)$ for some $N$ as follows:
 \begin{equation}
 \xymatrix{
  M_2(\widehat{\integer}) \times \upper \ar[d]_{p_N \times \id_\upper} \ar[r]^f & \comp \\
  M_2(\integer / N \integer)  \times \upper \ar[ur] &
}
 \end{equation}
\end{enumerate}
\end{defn}

\begin{rem}
In relation to the \emph{Bost-Connes system} \cite{Connes_Marcolli_Ramachandran}, the concept of Witt deformation is based on a slight modification of the arithmetic subalgebra constructed in Definition 2.22 \cite{Connes_Marcolli_Ramachandran}. Modulo this modification, our modularity theorem in \S \ref{s4.3} can be seen (almost) as the claim that the arithmetic subalgebra given loc.\ cit.\ (or its commutative subalgebra) coincides with that constructed in \cite{Yalkinoglu}. For this proof to work well, however, we needed to modify the constructions in \cite{Connes_Marcolli_Ramachandran,Connes_Marcolli}; hence our construction is slightly different from there. Despite of this modification, it would be meaningful to have in mind the relationship between our result and the works on Bost-Connes systems \cite{Connes_Marcolli_Ramachandran,Laca_Larsen_Neshveyev,Yalkinoglu} particularly when we try to extend the results developed in this paper to number fields of higher degree; in particular the results of \cite{Laca_Larsen_Neshveyev,Yalkinoglu} and ours in \S \ref{s3} hold for arbitrary number fields. 
\end{rem}

\begin{ex}[$j$-function]
\label{j}
The $j$-function $j$ naturally defines the simplest Witt deformation family of modular functions, denoted by the same symbol $j: M_2(\widehat{\integer}) \times \upper \rightarrow \comp$: for each $(m, \tau) \in M_2(\widehat{\integer}) \times \upper$, 
\begin{eqnarray}
  j(m,\tau) &:=& j(\tau). 
\end{eqnarray}
The fact that this $j: M_2(\widehat{\integer}) \times \upper \rightarrow \comp$ indeed defines a Witt deformation follows readily from the definition of the action of $GL_2(\Ad_{\ratf,f}) = GL_2(\widehat{\integer}) GL_2^+(\ratf)$ on $j$, in particular, $j^u = j$ for $u \in GL_2(\widehat{\integer})$ (cf.\ \S 6.6, \cite{Shimura}). 
\end{ex}

\begin{ex}[Fricke functions]
\label{Fricke functions}
The \emph{Fricke functions} are modular functions $f_a: \upper \rightarrow \comp$ indexed by row vectors $a \in \ratf^2 / \integer^2 \setminus \{0\}$, and defined as follows (cf.\ pp.133 \cite{Shimura}): for $w_1, w_2$ with $w_1/w_2 \in \upper$, 
\begin{eqnarray}
  f_a (w_1/w_2) &=& \frac{g_2(w_1, w_2) g_3(w_1,w_2)}{\Delta(w_1, w_2)} \wp ( a_1 w_1 + a_2 w_2 ; w_1, w_2).
\end{eqnarray}
Each Fricke function $f_a$ naturally defines a Witt deformation family $\chi_a$ of modular functions by: (here we shall define $f_0 := j$.)
\begin{eqnarray}
 \chi_a: M_2(\widehat{\integer}) \times \upper &\rightarrow& \comp \\
 (m, \tau) &\longmapsto& f_{am} (\tau) 
\end{eqnarray}
(In some sense this $\chi_a$ ``deforms'' the Fricke functions $f_a$ by the action of $m \in M_2(\widehat{\integer})$ on the index $a \in \ratf^2 / \integer^2$.) The fact that this $\chi_a: M_2(\widehat{\integer}) \times \upper \rightarrow \comp$ indeed defines a Witt deformation of modular functions can be proved, again, using the definition of the action of $GL_2(\Ad_{\ratf,f}) = GL_2(\widehat{\integer}) GL_2^+(\ratf)$ on $f_a$'s: That is, for $u \alpha \in GL_2(\widehat{\integer}) GL_2^+(\ratf) = GL_2(\Ad_{\ratf,f})$, one has $f_a^{u\alpha} = f_{au} \circ \alpha$ (cf.\ \S 6.6, \cite{Shimura}). 
\end{ex}

We demonstrate in Theorem \ref{explicit modularity theorem} that these functions (Example \ref{j}, \ref{Fricke functions}) are typical and provide enough generators of \emph{modular vectors} (cf.\ Definition \ref{modular vector}), but we have yet another class of examples as follows, which is useful in our constructions (in Lemma \ref{rho}, \ref{enough}) of modular vectors:

\begin{ex}[characteristic functions]
\label{characteristic functions}
Let $N \geq 1$ be a positive integer and $S \subseteq M_2(\integer / N \integer)$ be a subset that is closed under the right action of $GL_2(\integer/ N \integer)$. Denoting by $\chi_S: M_2(\integer/ N \integer) \rightarrow \comp$ the characteristic function of $S$, let us define the function $f^{(S)}: M_2(\widehat{\integer}) \times \upper \rightarrow \comp$ as the composition:
\begin{equation}
 f^{(S)}: M_2(\widehat{\integer}) \times \upper \xrightarrow{\pi_1} M_2(\widehat{\integer}) \twoheadrightarrow M_2(\integer/ N \integer) \xrightarrow{\chi_S} \comp
\end{equation}
where $\pi_1$ denotes the projection onto the first component. Then this $f^{(S)}$ gives a Witt deformation, whose fiber $f_m: \upper \rightarrow \comp$ at each $m \in M_2(\widehat{\integer})$ is either $0$ or $1$, the constant functions. 
\end{ex}

\begin{defn}[modular vector]
\label{modular vector}
Let $f: M_2(\widehat{\integer}) \times \upper \rightarrow \comp$ be a Witt deformation of modular functions such that the fiber $f_m: \upper \rightarrow \comp$ at each $m \in M_2(\widehat{\integer})$ is defined as a finite value at $\tau_K \in \upper$ with $O_K = \integer \tau_K + \integer$. Then the \emph{modular vector} $\widehat{f} \in \comp^{I_K}$ \emph{associated to $f$ (at $K$)} is the function $\widehat{f}: I_K \rightarrow \comp$ defined by the following composition: 
\begin{equation}
\label{composition0}
\widehat{f}: I_K \hookrightarrow DR_K \xrightarrow{\simeq} \widehat{O}_K \times_{\widehat{O}_K^*} (\Ad_{K,f}^*/K^*) \xrightarrow{\simeq} Lat^1_K \hookrightarrow Lat^2_\ratf \xrightarrow{\simeq} \Gamma \backslash (M_2(\widehat{\integer}) \times \upper) \xrightarrow{f} \comp. 
\end{equation}
Note that $f$ induces a function on $\Gamma \backslash (M_2(\widehat{\integer}) \times \upper)$ thanks to the second condition of Witt deformation, and $\widehat{f}$ in fact takes its values in $K^{ab}$, which follows from Lemma \ref{component of modular vector} below. Modular vectors in this sense themselves constitute a $\ratf$-algebra; to make it a $K$-algebra, we define \emph{the $K$-algebra $M_K$ of modular vectors} as follows (and abusively, we shall call the elements of $M_K$ as well \emph{modular vectors}):
\begin{eqnarray}
\label{M_K}
  M_K &:=& K \otimes_\ratf \bigl \{ \widehat{f} \in (K^{ab})^{I_K} \mid \textrm{$f: M_2(\widehat{\integer}) \times \upper \rightarrow \comp$ is a Witt deformation.}  \bigr\}.
\end{eqnarray}
\end{defn}

\begin{rem}
Instead, we may allow the fibers $f_m$ of Witt deformations $f: M_2(\widehat{\integer}) \times \upper \rightarrow \comp$ to be modular functions in $K \otimes_\ratf \modular$; then the modular vectors $\widehat{f}$ constitute a $K$-algebra in themselves, which is equal to (\ref{M_K}). We may use the term ``Witt deformations'' in this extended sense too.
\end{rem}

Concerning modular vectors, we first see more explicitly the coefficients $\widehat{f}_\ida$ of modular vectors $\widehat{f} = (\widehat{f}_\ida)$ in terms of the values of the Witt deformation $f: M_2(\widehat{\integer}) \times \upper \rightarrow \comp$ of modular functions: 

\begin{lem}
\label{component of modular vector}
Let $f: M_2(\widehat{\integer}) \times \upper \rightarrow \comp$ be a Witt deformation of modular functions. Then, for each $\ida \in I_K$, the component $\widehat{f}_\ida \in \comp$ at $\ida$ of the associated modular vector $\widehat{f} = (\widehat{f}_\ida) \in \comp^{I_K}$ is in fact in $K^{ab}$ and given by the following equation:
\begin{eqnarray}
 \widehat{f}_\ida  &=& f(m_\ida, \tau_{\ida^{-1}});
\end{eqnarray}
where, on one hand, if we denote as $\ida^{-1} = \integer w_1 + \integer w_2$ with $w_1 / w_2 \in \upper$, we define $\tau_{\ida^{-1}} := w_1/ w_2$; on the other hand, $m_\ida \in M_2(\widehat{\integer})$ is such that the following diagram commutes with some $\alpha \in GL_2^+(\realf)$ (i.e.\ such that $\integer^2 \cdot \alpha \cdot (i, 1)^t = \ida^{-1}$; cf.\ \S \ref{s2.2}):
\begin{equation}
 \xymatrix{
  \ratf^2/\integer^2 \ar[d]_{q_\tau} \ar[r]^{m_\ida} & \ratf^2/\integer^2 \ar[r]^{\alpha \cdot (i,1)^t} & \ratf \ida^{-1} / \ida^{-1} \ar[d]^{=} \\
   K/ O_K \ar[rr] && K \ida^{-1}/\ida^{-1}
  }
\end{equation}
where $K/O_K \rightarrow K\ida^{-1} / \ida^{-1}$ is a natural inclusion. 
\end{lem}
\begin{proof}
The proof is just based on a careful chase of the composition (\ref{composition0}). Firstly, the composition $I_K \hookrightarrow DR_K \simeq \widehat{O}_K \times_{\widehat{O}_K^*} (\Ad_{K,f}^* / K^*)$ sends the ideal $\ida \in I_K$ (denoting $\ida = \rho \widehat{O}_K \cap K$ for some $\rho \in \widehat{O}_K$) to $[\rho, \rho^{-1}] \in \widehat{O}_K \times_{\widehat{O}_K^*} (\Ad_{K,f}^* / K^*)$ (cf.\ Proposition 4.2 and Proposition 8.2, \cite{Yalkinoglu}); secondly, one sees that the further composition with the maps $\widehat{O}_K \times_{\widehat{O}_K^*} (\Ad_{K,f}^* / K^*) \rightarrow Lat^1_K \rightarrow Lat^2_\ratf$ sends $[\rho, \rho^{-1}]$ to $[\ida^{-1}, \phi_\ida] \in Lat^2_\ratf$, where $\phi_\ida: \ratf^2 / \integer^2 \rightarrow \ratf \ida^{-1} / \ida^{-1}$ is given by the following composition (cf.\ \S \ref{s2.2}): 
\begin{equation}
 \xymatrix{
 \phi_\ida: \ratf^2/\integer^2 \ar[r]^{q_\tau} & K/O_K \ar[r] & K \ida^{-1}/\ida^{-1} \ar[r]^{=}& \ratf \ida^{-1} / \ida^{-1};
}
\end{equation}
where the second arrow $K/O_K \rightarrow K \ida^{-1}/ \ida^{-1}$ is a natural inclusion. Therefore, by the construction of the isomorphism $Lat^2_\ratf \simeq \Gamma \backslash (M_2(\widehat{\integer}) \times \upper)$ (cf.\ \S \ref{s2.2}), this $[\ida^{-1}, \phi_\ida] \in Lat^2_\ratf$ indeed corresponds to $[m_\ida, \tau_{\ida^{-1}}] \in  \Gamma \backslash (M_2(\widehat{\integer}) \times \upper)$ for the $m_\ida \in M_2(\widehat{\integer})$ and $\tau_{\ida^{-1}} \in \upper$ described above. This completes the proof. 
\end{proof}

\begin{ex}[the modular vector for $j$]
Let $j: M_2(\widehat{\integer}) \times \upper \rightarrow \comp$ be the Witt deformation of modular functions given in Example \ref{j}. Then Lemma \ref{component of modular vector} shows that the associated modular vector $\widehat{j}$ is given by $\widehat{j}_\ida = j(\ida^{-1})$. Therefore, this $\widehat{j}$ is equal to $\iota$ constructed in Proposition \ref{witt vector from j function}, and as observed there, this vector $\widehat{j} = \iota$ is a member of $\intwitt$, hence, of the $K$-algebra $E_K$ in particular. This example shows that \emph{some} modular vector (i.e.\ $\widehat{j} \in M_K$) indeed defines an algebraic Witt vector (i.e.\ $\widehat{j} = \iota \in E_K$). 
\end{ex}

The goal of the next subsection is to show that this example is a special case of a more general relationship between modular vectors and algebraic Witt vectors: we prove that these two concepts actually coincide. 

\subsection{The modularity theorem}
\label{s4.3}
This section proves that the $K$-algebra $M_K$ of modular vectors coincides with the $K$-algebra $E_K$ of algebraic Witt vectors, which we call the \emph{modularity theorem}. In other words, this theorem claims that every modular vector $\widehat{f}$ defines an algebraic Witt vector, i.e.\ $\widehat{f} \in E_K$, and conversely, every algebraic Witt vector $\xi \in E_K$ is realized as a modular vector, i.e.\ $\xi = \widehat{f}$ for some Witt deformation $f: M_2(\widehat{\integer}) \times \upper \rightarrow \comp$ of modular functions. To be precise:

\begin{thm}[modularity theorem]
\label{modularity theorem}
We have the following identity as $K$-subalgebras of $(K^{ab})^{I_K}$:
\begin{eqnarray}
 E_K &=& M_K.
\end{eqnarray}
\end{thm}

Our proof of this theorem consists of two steps. We first prove the inclusion $M_K \subseteq E_K$, which is based on Corollary \ref{characterization of E_K} and Shimura's reciprocity law. After that, we  prove the converse inclusion $E_K \subseteq M_K$, which is based on the characterization of $E_K$ by simple conditions due to Neshveyev, Theorem 10.1, \S 10 \cite{Yalkinoglu}. 

To be specific, recall that Corollary \ref{characterization of E_K} proved that the $K$-algebra $E_K$ is identified with that of locally constant $K^{ab}$-valued $G_K^{ab}$-equivariant functions on $DR_K$, i.e.\ $E_K = \Hom_{G_K^{ab}}(DR_K, K^{ab})$; in the proof of the inclusion $M_K \subseteq E_K$, we see that every modular vector $\widehat{f}$ defines a locally constant ($K^{ab}$-valued) $G_K^{ab}$-equivariant function on $DR_K$, by which we deduce the inclusion $M_K \subseteq E_K$. On the other hand, the $K$-algebra $E_K$ was characterized by Neshveyev, \S 10 \cite{Yalkinoglu} as the $K$-subalgebra $E$ of $C(DR_K)$--- the $\comp$-algebra of continuous functions on $DR_K$--- that satisfies the following conditions (cf.\ pp.408, \cite{Yalkinoglu}; note that $DR_K$ is homeomorphic to $Y_K$ in \cite{Yalkinoglu}): 
\begin{enumerate}
 \item every function in $E$ is locally constant on $DR_K$;
 \item $E$ separates the points of $DR_K$; 
 \item $E$ contains the idempotents $\rho^\ida$ for each $\ida \in I_K$ (cf.\ Lemma \ref{rho});
 \item every function in $E$ is $K^{ab}$-valued and $G_K^{ab}$-equivariant. 
\end{enumerate}
In other words, if $E$ is a $K$-subalgebra of $E_K$ and satisfies the second and the third conditions here, one actually has the identity $E = E_K$. Therefore, after the proof of the inclusion $M_K \subseteq E_K$, the remained task for the proof of the identity $E_K = M_K$ is only to prove that $M_K$ satisfies the second and the third conditions here. Our proof of the target identity $E_K = M_K$ will follow this line of argument. 

Now let us start proving the inclusion $M_K \subseteq E_K$. To this end, we prepare the following lemma, which is used to prove that modular vectors $\widehat{f}$ define $K^{ab}$-valued $G_K^{ab}$-equivariant functions on $DR_K$. In what follows, recall that every fractional ideal $\Lambda$ of $K$ can be given as $\Lambda = t \widehat{O}_K \cap K$ for some $t \in \Ad_{K,f}$; then, for such $\Lambda$ and each $s \in \Ad_{K,f}$, let us denote $s^{-1} \Lambda := s^{-1}t \widehat{O}_K \cap K$. 

\begin{lem}
\label{G_K equivariance}
Let $(\Lambda,\phi) \in Lat^1_K$ be a $K$-lattice such that $\Lambda$ is a fractional ideal of $K$ and corresponds to $[m, \tau] \in \Gamma \backslash (M_2(\widehat{\integer}) \times \upper)$; also let $s \in \Ad_{K,f}^* / K^*$ and $(s^{-1} \Lambda, s^{-1} \phi) \in Lat^1_K$ corresponds to $[m_s, \tau_s]$. Moreover, let $f: M_2(\widehat{\integer}) \times \upper \rightarrow \comp$ be a Witt deformation of modular functions. Then $f_m (\tau) \in K^{ab}$, and the following identity holds:
\begin{eqnarray}
 \bigl (f_m (\tau) \bigr)^{[s]} &=& f_{m_s} (\tau_s). 
\end{eqnarray}
\end{lem}
\begin{proof}
Throughout this proof, let us denote $\Lambda_s := s^{-1} \Lambda$ and $\phi_s := s^{-1} \phi$ for short. Moreover let $\Lambda = \integer w_1 + \integer w_2 = \integer^2 \cdot \alpha \cdot (i, 1)^t$ and $\Lambda_s = \integer w_1' + \integer w_2' = \integer^2 \cdot \alpha_s \cdot (i, 1)^t$ with $\tau = w_1 / w_2, \tau_s = w_1'/ w_2' \in \upper$ and $\alpha, \alpha_s \in GL_2^+(\realf)$; also let $\beta \in GL_2^+ (\ratf)$ be such that $(w_1, w_2)^t = \beta (w_1', w_2')^t$, whence $\alpha_s = \beta^{-1} \alpha$ and $\tau= \alpha (i), \tau_s = \alpha_s (i)$.  
Define the embedding $q_\Lambda: \Ad_{K,f} \rightarrow M_2(\Ad_{\ratf,f})$ with respect to the basis of $\Lambda$ in the sense of \S 4.4 \cite{Shimura} that $q_\Lambda (\mu) (w_1, w_2)^t = (\mu w_1, \mu w_2)^t$ for $\mu \in K$. 
Then, for each prime $p$, we have:
\begin{eqnarray}
 \integer_p^2 \left (\begin{array}{c} w_1' \\ w_2' \end{array} \right) &=& (\Lambda_s)_p \\
 &=& \Lambda_p s_p^{-1} \\
 &=&  \integer_p^2 \left (\begin{array}{c} w_1 s_p^{-1} \\ w_2 s_p^{-1} \end{array} \right) \\
 &=&  \integer_p^2 q_\Lambda (s_p)^{-1} \left (\begin{array}{c} w_1 \\ w_2 \end{array} \right) \\
 &=&  \integer_p^2 q_\Lambda (s_p)^{-1} \beta \left (\begin{array}{c} w_1' \\ w_2' \end{array} \right). 
\end{eqnarray}
Since this holds for each prime $p$, we have some $u \in GL_2(\widehat{\integer})$ such that $q_\Lambda (s)^{-1} \beta = u$; therefore, $q_\Lambda (s)^{-1} = u \beta^{-1}$. With this in mind, we first compute $(f_m (\tau))^{[s]}$. By Shimura's reciprocity law and $q_\Lambda(s)^{-1} = u\beta^{-1}$, we have:
\begin{eqnarray}
  (f_m(\tau))^{[s]}  &=& f_{m}^{q_\Lambda(s)^{-1}} (\tau) \\
  &=& f_{mu} \circ \beta^{-1} (\tau). 
\end{eqnarray}
To prove that the last one is equal to $f_{m_s}(\tau_s)$, consider the following composition:
\begin{equation}
\label{composition}
\xymatrix{
\ratf^2/\integer^2 \ar[r]^m & \ratf^2/\integer^2 \ar[r]^{\alpha \cdot (i,1)^t} & \ratf \Lambda / \Lambda \ar[r]^{s^{-1}} & \ratf \Lambda_s / \Lambda_s
};
\end{equation}
here the last arrow is defined by the following; suppose $\Lambda = t \widehat{O}_K \cap K$ for $t \in \Ad_{K,f}$. Then:
\begin{equation}
\xymatrix{
\ratf \Lambda / \Lambda \ar[d]_{=} \ar[r]^{s^{-1}} & \ratf \Lambda_s / \Lambda_s \ar[d]^{=} \\
K \Lambda/ \Lambda \ar[d]_{\simeq} & K \Lambda_s / \Lambda_s \ar[d]^{\simeq} \\
\Ad_{K,f} / t \widehat{O}_K \ar[r]_{s^{-1}} & \Ad_{K,f} / s^{-1}t\widehat{O}_K 
}
\end{equation}
Recall that $q_\Lambda (s)$ is defined with respect to the basis of $\Lambda = \integer^2 \cdot \alpha \cdot (i, 1)^t$. Therefore, we have the following commutative diagram: (recall also that $\Lambda_s = \integer^2 \cdot \alpha_s \cdot (i,1)^t$ with $\alpha_s = \beta^{-1} \alpha$; and we write just by $\alpha: \ratf^2/\integer^2 \rightarrow \ratf \Lambda/\Lambda$ to mean the composition $\alpha \cdot (i, 1)^t: \ratf^2/\integer^2 \rightarrow \ratf^2 \alpha/\integer^2 \alpha \rightarrow \ratf \Lambda/\Lambda$.)
\begin{eqnarray}
\xymatrix{
 \ratf \Lambda/ \Lambda \ar[r]^{s^{-1}} \ar[d]_{\alpha ^{-1}} & \ratf \Lambda_s / \Lambda_s \\
 \ratf^2 / \integer^2 \ar[r]_{q_\Lambda (s)^{-1}} & \ratf^2 / \integer^2 \beta^{-1} \ar[u]_{\alpha} 
 }
\end{eqnarray}
This diagram shows that the composition (\ref{composition}) is equal to the following:
\begin{equation}
\label{composition2}
\xymatrix{
\ratf^2/\integer^2 \ar[r]^m & \ratf^2/\integer^2 \ar[r]^{q_\Lambda (s)^{-1}} & \ratf^2 / \integer^2 \beta^{-1} \ar[r]^{\alpha} & \ratf \Lambda_s / \Lambda_s
}.
\end{equation}
Therefore, $m_s \alpha_s = m q_\Lambda (s)^{-1} \alpha = m u \beta^{-1} \alpha$; and hence we have: (note that $\tau = \alpha (i)$ and $\tau_s = \alpha_s(i)$ by definition above.)
\begin{eqnarray}
 f_{m u} \circ \beta^{-1} (\tau) &=& f_{mu} \circ \beta^{-1} \alpha (i) \\
 &=& f_{m_s} (\alpha_s(i)) \\
 &=& f_{m_s} (\tau_s). 
\end{eqnarray}
Consequently, we proved that $(f_m(\tau))^{[s]} = f_{m_s}(\tau_s)$ as requested. This completes the proof. 
\end{proof}

With this lemma, we first prove the one-side inclusion for the target equality $E_K = M_K$:

\begin{lem}
\label{one-sided inclusion}
We have the following inclusion:
\begin{eqnarray}
 M_K &\subseteq& E_K.
\end{eqnarray}
\end{lem}
\begin{proof}
Let $f: M_2(\widehat{\integer}) \times \upper \rightarrow \comp$ be a Witt deformation of modular functions with $\widehat{f} \in M_K$ being the associated modular vector. Firstly we see that the associated function $\widehat{f}: DR_K \rightarrow \comp$ is locally constant. (Here recall that $\widehat{f}: I_K \rightarrow K^{ab} \in M_K$ is defined as a restriction of a (continuous) function $\widehat{f}: DR_K \rightarrow \comp$ onto the dense subset $I_K \subseteq DR_K$; cf.\ Definition \ref{modular vector}.) Since $f: M_2(\widehat{\integer}) \times \upper \rightarrow \comp$ by definition factors through some projection $p_N: M_2(\widehat{\integer}) \twoheadrightarrow M_2(\integer / N \integer)$, this function $\widehat{f}: DR_K \rightarrow \comp$ also factors through the projection $DR_K \twoheadrightarrow DR_{NO_K}$; this readily follows from the isomorphism $DR_{NO_K} \simeq (O_K / NO_K) \times_{(O_K / NO_K)^*} C_{NO_K}$, cf.\ Proposition 7.2 \cite{Yalkinoglu}. Hence, $\widehat{f}: DR_K \rightarrow \comp$ is indeed locally constant on $DR_K$. 

Secondly, to see that $\widehat{f}: DR_K \rightarrow \comp$ is $G_K^{ab}$-equivariant, we remark that the action of $[s] \in G_K^{ab}$ for $s \in \Ad_{K,f}^*$ on the element of $DR_K$ that corresponds to $[\rho, t] \in \widehat{O}_K \times_{\widehat{O}_K^*} (\Ad_{K,f}^* / K^*)$ under the homeomorphism $\psi: DR_K \rightarrow \widehat{O}_K \times_{\widehat{O}_K^*} (\Ad_{K,f}^* / K^*)$ corresponds to $[\rho, s^{-1} t]$ (i.e.\ not $[\rho, st]$); that is, we have $\bigl(\psi^{-1}(\rho, t)\bigr)^{[s]} = \psi^{-1} (\rho, s^{-1} t)$; this follows from the construction of the homeomorphism $\Psi: \widehat{O}_K \times_{\widehat{O}_K^*} G_K^{ab} \rightarrow DR_K$ of \S 8.1 \cite{Yalkinoglu}. Thus, if $\psi^{-1} (\rho, t) \in DR_K$ corresponds to $[m, \tau] \in \Gamma \backslash (M_2(\widehat{\integer}) \times \upper)$ under the correspondence $DR_K \rightarrow \Gamma \backslash (M_2(\widehat{\integer}) \times \upper)$ (cf.\ (\ref{composition0})), then $\bigl(\psi^{-1}(\rho, t)\bigr)^{[s]} = \psi^{-1}(\rho, s^{-1}t) \in DR_K$ corresponds to $[m_s, \tau_s]$ (in the sense of the notation in Lemma \ref{G_K equivariance}). Therefore, we have:
\begin{eqnarray}
 \widehat{f} ( \bigl(\psi^{-1}(\rho, t)\bigr)^{[s]} )  &=& \widehat{f} (\psi^{-1}(\rho, s^{-1}t)) \\
 &=& f_{m_s} (\tau_s) \\
 &=& \bigl(f_m (\tau)\bigr)^{[s]} \\
 &=& \widehat{f} (\psi^{-1}(\rho, t))^{[s]}. 
\end{eqnarray}
In the third equality, we used Lemma \ref{G_K equivariance}; this shows that $\widehat{f}: DR_K \rightarrow \comp$ is indeed $G_K^{ab}$-equivariant. Therefore $\widehat{f}$ is $K^{ab}$-valued locally constant $G^{ab}_K$-equivariant function on $DR_K$, hence $\widehat{f} \in E_K$. This completes the proof. 
\end{proof}

Now we have the inclusion $M_K \subseteq E_K$. Thus, in order to prove our target identity $E_K = M_K$, it suffices to prove that $M_K$ satisfies the above-mentioned two conditions (2) and (3). In the following we first prove the third condition (3), and then prove the second (2), which will complete the proof of our target theorem. 

\begin{lem}
\label{rho}
 For each $\ida \in I_K$, let us define $\rho^\ida \in (K^{ab})^{I_K}$ as follows:
 \begin{eqnarray}
  \rho^\ida (\idb) &:=& \left\{ \begin{array}{ll} 1 & \textrm{if $\ida \mid \idb$} \\ 0 & \textrm{otherwise} \end{array} \right.
 \end{eqnarray}
 Then $\rho^\ida \in M_K$ for every $\ida \in I_K$. 
\end{lem}
\begin{proof}
 This function $\rho^\ida$ corresponds to the characteristic function of $\ida DR_K$, which we now show belongs to $M_K$. Let us fix $s_\ida \in \widehat{O}_K$ such that $\ida = s_\ida \widehat{O}_K \cap K$. Under the homeomorphism $DR_K \simeq \widehat{O}_K \times_{\widehat{O}_K^*} (\Ad_{K,f}^* / K^*)$ of \cite{Yalkinoglu}, the subspace $\ida DR_K \subseteq DR_K$ corresponds to $\ida \widehat{O}_K \times_{\widehat{O}_K^*} (\Ad_{K,f}^* / K^*)$; and note that, under the further composition $\widehat{O}_K \times_{\widehat{O}_K^*} (\Ad_{K,f}^* / K^*) \rightarrow Lat^1_K \rightarrow Lat^2_\ratf \rightarrow \Gamma \backslash (M_2(\widehat{\integer}) \times \upper)$, this subspace is eventually mapped into $\Gamma \backslash \bigl(q_\tau (s_\ida) M_2(\widehat{\integer}) \times \upper \bigr)$, where $q_\tau: \Ad_{K,f} \rightarrow M_2(\Ad_{\ratf,f})$ denotes the embedding determined by $O_K = \integer \tau + \integer$. Let us define the function $f^\ida: M_2(\widehat{\integer}) \times \upper \rightarrow \comp$ as the characteristic function of the subspace $q_\tau (s_\ida) M_2(\widehat{\integer}) \times \upper \subseteq M_2(\widehat{\integer}) \times \upper$. It is clear that $f^\ida$ satisfies the conditions (1), (2) of Witt deformations of modular functions. 
 In order to see that $f^\ida$ also satisfies the condition (3), we note that $f^\ida$ factors through $p_{N(\ida)}: M_2(\widehat{\integer}) \twoheadrightarrow M_2(\integer/ N(\ida) \integer)$, where $N(\ida)$ is the absolute norm of $\ida$. In fact, this follows from the fact that the kernel of $p_{N(\ida)}$ is $N(\ida) M_2(\widehat{\integer})$, which is included in $q_\tau (s_\ida) M_2(\widehat{\integer})$. From these facts, we find that $f^\ida$ is in fact a Witt deformation of modular functions; and by construction, it is clear that $\widehat{f}^\ida = \rho^\ida$. This implies that $\rho^\ida \in M_K$ as requested. This completes the proof. 
\end{proof}

\begin{lem}
\label{enough}
 The functions in $M_K$ are enough to separate the points of $DR_K$. 
\end{lem}
\begin{proof}
Let $x \neq y \in DR_K$ be distinct points. Since $DR_K$ is the inverse limit of $DR_\idf$'s for $\idf \in I_K$, there exists $\idf \in I_K$ such that $x$ and $y$ still represent distinct elements of $DR_\idf$ under the projection $DR_K \twoheadrightarrow DR_\idf$; also since the ideals $N O_K$ for positive integers $N$ are cofinal in $I_K$, we may suppose $\idf = N O_K$.  In this proof let us denote by $\bar{x}, \bar{y} \in DR_\idf$ those in $DR_\idf$ represented by $x, y \in DR_K$ under the projection $DR_K \twoheadrightarrow DR_\idf$. Here recall the decomposition $DR_\idf = \coprod_{\idd \mid \idf} C_{\idf/\idd}$ (see (4.4) in pp.394 \cite{Yalkinoglu}). If $\bar{x}$ and $\bar{y}$ belong to distinct components, then the characteristic functions defined above are enough to separate $x, y$. 

So we suppose that $\bar{x}, \bar{y}$ belong to the same component $C_{\idf/\idd}$ for some $\idd \mid \idf$ but represent distinct elements. This implies that there exist $[\rho, s], [\rho, t] \in \widehat{O}_K \times_{\widehat{O}_K^*} (\Ad_{K,f}^* / K^*)$ such that $\idd = \rho \widehat{O}_K \cap K$ and $[\rho, s], [\rho, t]$ correspond respectively to $\bar{x}, \bar{y}$ under the composition $\widehat{O}_K \times_{\widehat{O}_K^*} (\Ad_{K,f}^* / K^*) \simeq DR_K \twoheadrightarrow DR_\idf$; also, by $\bar{x} \neq \bar{y} \in C_{\idf/\idd} \subseteq DR_\idf$, the actions of $[s], [t] \in Gal(K^{ab}/ K)$ are not identical on $K_{\idf/\idd}$. Thus by the corollary in \S 3, Chapter 10  \cite{Lang}, we can find $a \in (N^{-1} \integer^2 / \integer^2) \cdot q_{\tau_K} (\rho) \simeq N^{-1}\idd / O_K$ so that $(f_a(\tau_K))^{[s]^{-1}} \neq (f_a (\tau_K))^{[t]^{-1}}$. (To be precise, when $K = \ratf(i)$ or $K= \ratf(\sqrt{-3})$, we replace $f_a$'s in Example \ref{Fricke functions} with the functions $f_a^2$'s or $f_a^3$'s given in \S 6.1 \cite{Shimura}.) With these data, let us now define a function $f: M_2(\widehat{\integer}) \times \upper \rightarrow \comp$ as follows: (Here $q_{\tau_K}: \Ad_{K,f} \rightarrow M_2(\Ad_{\ratf,f})$ denotes the embedding determined by $O_K = \integer \tau_K + \integer$.)
\begin{eqnarray}
 f (m, \tau) &:=& \left\{ \begin{array}{ll} 0 & \textrm{if $m \not \in q_{\tau_K} (\rho) M_2(\widehat{\integer}) \subset M_2(\widehat{\integer})$} \\ f_{a m'} (\tau) & \textrm{if $m = q_{\tau_K} (\rho) m'$ for $\exists ! m' \in M_2(\widehat{\integer})$} \end{array} \right.
\end{eqnarray}
By the same argument as above and $a \in (N^{-1} \integer^2/ \integer^2) \cdot q_{\tau_K}(\rho)$, we see that this function $f$ factors through the quotient $M_2(\widehat{\integer}) \twoheadrightarrow M_2(\integer/ N \integer)$; and  $f$ defines a Witt deformation of modular functions. We now show that the modular vector $\widehat{f}$ associated to this $f$ separates $x$ and $y$. Note that, since $f$ factors through $M_2(\widehat{\integer}) \twoheadrightarrow M_2(\integer / N \integer)$, the function $\widehat{f}: DR_K \rightarrow \comp$ also factors through $DR_K \twoheadrightarrow DR_\idf = DR_{NO_K}$. (Again recall that we have $DR_{NO_K} \simeq (O_K / N O_K) \times_{(O_K / N O_K)^{*}} C_{NO_K}$; cf.\ Proposition 7.2 \cite{Yalkinoglu}.) Then, by construction and the $G^{ab}_K$-equivariance (cf.\ Lemma \ref{one-sided inclusion}) of this function $\widehat{f}:DR_K \simeq \widehat{O}_K \times_{\widehat{O}_K^*} (\Ad_{K,f}^* / K^*) \rightarrow \comp$, we have:
\begin{equation}
\widehat{f}(x) = \widehat{f}(\rho, s) = (\widehat{f}(\rho, 1))^{[s]^{-1}} 
= (f_a (\tau_K))^{[s]^{-1}} \neq (f_a (\tau_K))^{[t]^{-1}} 
= (\widehat{f}(\rho, 1))^{[t]^{-1}}
= \widehat{f}(\rho, t) 
= \widehat{f}(y). 
\end{equation}
This completes the proof. 
\end{proof}

\begin{cor}
 We have the following inclusion:
 \begin{eqnarray}
  E_K &\subseteq& M_K; 
 \end{eqnarray}
 hence $E_K = M_K$. 
\end{cor}
\begin{proof}
In the above developments, we have proved that $M_K$ is a $K$-subalgebra of $E_K$, separates the points of $DR_K$, and contains the idempotents $\rho^\ida$ for $\ida \in I_K$. Therefore, by Theorem 10.1 \cite{Yalkinoglu}, we conclude the target identity $E_K = M_K$. 
\end{proof}

\subsection{Explicit generators}
\label{s4.4}
To further clarify the description of $E_K$ as $M_K$ (Theorem \ref{modularity theorem}), we conclude this paper 
by giving a set of explicit generators of $M_K = E_K$. To be more specific we demonstrate that the galois objects 
$X_{\Xi(N)}$ generated by the modular vectors $\Xi(N) := \{\widehat{\chi}_a \in M_K \mid a \in N^{-1} \integer^2/\integer^2 \}$ $(N \geq 1)$ are cofinal among the galois objects of $\C_K$; in particular, as we discuss below, this implies that $E_K = M_K$ is generated by the modular vectors $\widehat{\chi}_a$ coming from the Fricke
 functions $f_a$. (For $K=\ratf(i), \ratf(\sqrt{-3})$, we replace $f_a$'s with $f_a^2$'s and $f_a^3$'s in \S 6.1 \cite{Shimura}.)
 
 Throughout this subsection, we abusively denote for brevity as $f_a = \chi_a: M_2(\widehat{\integer}) \times \upper \ni (m, \tau) \mapsto f_{am}(\tau) \in \comp$, thus $\widehat{f}_a = \widehat{\chi}_a$; again recall that we put $f_0 := j$. To prove the following theorem is our major goal in this subsection: 
 
 \begin{thm}[Fricke functions are enough to generate $E_K$]
 \label{explicit modularity theorem}
 The $K$-algebra $E_K$ is generated as a $K$-algebra by the modular vectors $\widehat{f}_a$ for $a \in \ratf^2/\integer^2$. That is:
 \begin{eqnarray}
  E_K &=& K[ \widehat{f}_a \mid a \in \ratf^2/\integer^2]. 
 \end{eqnarray}
 \end{thm}
  
The proof of this theorem is based on three lemmas: With the first two lemmas (Lemma \ref{equal gcds}, Lemma \ref{subdivision}) we see that, for every positive integer $N$, the monoid congruence $\ida \sim_{NO_K} \idb$ on $I_K$ (cf.\ (\ref{ideal congruence})) coincides with the monoid congruence $\ida \equiv_{\Xi(N)} \idb$ (cf.\ \S \ref{s3.1}), which implies that the galois objects $X_{\Xi(N)}$ are cofinal in all the galois objects in $\C_K$; the last lemma (Lemma \ref{last lemma}) then proves that $X_{\Xi(N)}$ is generated by $\Xi(N)$ (less than the orbit $I_K \Xi(N)$). Since $E_K$ is the direct limit (i.e.\ the union) of the galois objects of $\C_K$, hence of $X_{\Xi(N)}$'s, these three lemmas will conclude the target theorem. 

We start with the following lemma:

\begin{lem}
\label{equal gcds}
Let $N$ be a positive integer. Then, for any $\ida, \idb \in I_K$ (not necessarily prime to $N$), $\ida \equiv_{\Xi(N)} \idb$ implies the equality $(\ida, NO_K) = (\idb, NO_K)$ of greatest common divisors. 
\end{lem}
\begin{proof}
Assume $(\ida, NO_K) \neq (\idb, NO_K)$, from which we deduce $\ida \not \equiv_{\Xi(N)} \idb$. By this assumption we have $\idp \mid NO_K$ such that $v_\idp (\ida) \neq v_\idp (\idb)$, and without loss of generality, we may assume $v_\idp (\ida) < v_\idp (\idb)$ and $v_\idp (\ida) < v_\idp (NO_K)$. We need to prove that there exists $a \in N^{-1} \integer^2 / \integer^2$ such that $\psi_\ida \widehat{f}_a \neq \psi_\idb \widehat{f}_a$, whence $\ida \not \equiv_{\Xi(N)} \idb$; so, more specifically, we prove $\widehat{f}_a (\idp^e \ida) \neq \widehat{f}_a (\idp^e \idb)$ for some $a \in N^{-1} \integer^2/\integer^2$, where $e := \max \{0, v_\idp(NO_K) - v_\idp (\idb)\} < v_\idp(NO_K) - v_\idp(\ida)$ .  

To prove $\widehat{f}_a (\idp^e \ida) \neq \widehat{f}_a (\idp^e \idb) \in K^{ab}$ for some $a \in N^{-1} \integer^2 / \integer^2$, it suffices to find $a \in N^{-1} \integer^2/\integer^2$ such that $\widehat{f}_a (\idp^e \ida)$ and $\widehat{f}_a (\idp^e \idb)$ generate different fields over $H_K$. To this end recall Corollary in \S 3, Chapter 10, \cite{Lang}, which shows that the ray class field $K_{\idf}$ over $K$ with conductor $\idf := NO_K / \idc \in I_K$ (for $\idc \mid NO_K$) can be given as: 
\begin{eqnarray}
 K_{NO_K / \idc} &=&  H_K\bigl(f_a (\tau_K) \mid a \in (N^{-1}\integer^2 / \integer^2)\cdot q_{\tau_K} (\rho_\idc) \bigr);
\end{eqnarray}
where $\rho_\idc \in \widehat{O}_K$ is taken so that $\idc = \rho_\idc \widehat{O}_K \cap K$; moreover, we can choose a single $a_\idc \in (N^{-1} \integer^2/\integer^2) \cdot q_{\tau_K}(\rho_\idc)$ so that $ K_{NO_K /\idc} = H_K(f_{a_\idc} (\tau_K))$ (cf.\ pp.135 \cite{Lang}). Applying this fact to $\idc := (\idp^e \ida, NO_K)$ in particular, we fix such $a_\idc =: a' \in (N^{-1} \integer^2/\integer^2) \cdot q_{\tau_K}(\rho)$, where we put $\rho := \rho_\idc = \rho_{(\idp^e \ida, NO_K)}$. 

Then we take $a \in N^{-1}\integer^2/\integer^2$ so that $a \cdot q_{\tau_K}(\rho) = a'$, with which we  prove that $H_K(\widehat{f}_a (\idp^e \ida)) \neq H_K (\widehat{f}_a(\idp^e \idb))$ by showing that $\idp$ is ramified in the former, while unramified in the latter. (This will complete the proof of this lemma.) To see this, let us take $[\rho, s], [\lambda, t] \in \widehat{O}_K \times_{\widehat{O}_K^*} (\Ad_{K,f}^* / K^*)$ such that the images of $[\rho, s], [\lambda, t]$ under the composition $\widehat{O}_K \times_{\widehat{O}_K^*} (\Ad_{K,f}^* / K^*) \simeq DR_K \twoheadrightarrow DR_{NO_K}$ are equal to $[\idp^e \ida], [\idp^e \idb] \in DR_{NO_K}$ respectively; here $\rho$ is the same as taken in the above. In particular $(\idp^e \ida, NO_K) = \rho \widehat{O}_K \cap K$ and $(\idp^e \idb, NO_K) = \lambda \widehat{O}_K \cap K$. As discussed in Lemma \ref{enough} too, recall also that the function $\widehat{f}_a: DR_K \rightarrow \comp$ factors through $DR_K \twoheadrightarrow DR_{NO_K}$ for $a \in N^{-1}\integer^2/\integer^2$. Then, with these data we have: 
\begin{eqnarray}
\widehat{f}_a (\idp^e \ida) &=& \widehat{f}_a (\rho, s) \\
&=& \bigl(f_{a q_{\tau_K}(\rho)}(\tau_K) \bigr) ^{[s]^{-1}}; \\
\widehat{f}_a (\idp^e \idb) &=& \widehat{f}_a (\lambda, t) \\
&=& \bigl(f_{a q_{\tau_K}(\lambda)}(\tau_K) \bigr) ^{[t]^{-1}}. 
\end{eqnarray}
Thus we have $H_K (\widehat{f}_a (\idp^e \ida)) = H_K (f_{aq_{\tau_K}(\rho)} (\tau_K))$ and $H_K (\widehat{f}_a (\idp^e \idb)) = H_K (f_{a q_{\tau_K}(\lambda)} (\tau_K))$. On one hand, by construction of $a \in N^{-1} \integer^2/\integer^2$, we have $H_K (\widehat{f}_a (\idp^e \ida)) = K_{NO_K / (\idp^e \ida, NO_K)}$; in particular, since $v_\idp (NO_K / (\idp^e \ida, NO_K)) > 0$, we see that $\idp$ ramifies in $H_K(\widehat{f}_a (\idp^e \ida)) = K_{NO_K / (\idp^e \ida, NO_K)}$. On the other hand, concerning $H_K (\widehat{f}_a (\idp^e \idb))$, we have $v_\idp (NO_K / (\idp^e \idb, NO_K)) = 0$; also, since $a q_{\tau_K}(\lambda) \in (N^{-1} \integer^2/\integer^2) \cdot q_{\tau_K} (\lambda)$, the field $H_K (\widehat{f}_a (\idp^e \idb))$ is a subfield of the ray class field $K_{NO_K / (\idp^e \idb, NO_K)}$ over $K$ whose conductor is not divisible by $\idp$. Therefore, $\idp$ is unramified in $H_K(\widehat{f}_a (\idp^e \idb))$ as requested. This completes the proof.
\end{proof}

\begin{lem}[modular description of the congruence $\sim_{NO_K}$]
\label{subdivision}
Let $N$ be a positive integer. Then $\ida \sim_{NO_K} \idb$ if and only if $\ida \equiv_{\Xi(N)} \idb$ for any $\ida, \idb \in I_K$. 
\end{lem}
\begin{proof}
We start with proving the only-if part. To this end, recall that for each $a \in N^{-1} \integer^2 / \integer^2$, the function $f_a (= \chi_a): M_2(\widehat{\integer}) \times \upper \rightarrow \comp$ factors through the projection $M_2(\widehat{\integer}) \twoheadrightarrow M_2(\integer / N \integer)$; and that the composition $DR_K \hookrightarrow \Gamma \backslash (M_2(\widehat{\integer}) \times \upper) \twoheadrightarrow \Gamma \backslash (M_2(\integer/ N \integer) \times \upper)$ factors through $DR_K \twoheadrightarrow DR_{N O_K}$. In other words, the modular vectors $\widehat{f}_a$'s for $a \in N^{-1} \integer^2 / \integer^2$ belong to $\Hom_{G_K}(DR_{NO_K}, \bar{K})$, i.e.\ the galois object $X_{N O_K}$ dual to $DR_{NO_K}$. Also $\ida \sim_{NO_K} \idb$ implies the identity $\psi_\ida = \psi_\idb$ on $X_{N O_K}$. Therefore, since $\widehat{f}_a \in X_{NO_K}$ for $a \in N^{-1} \integer^2/\integer^2$, we have $\psi_\ida \widehat{f}_a  = \psi_\idb \widehat{f}_a$ for every $a \in N^{-1} \integer^2 / \integer^2$, namely $\ida \equiv_{\Xi(N)} \idb$. 

Next we prove the if part along the same line as Lemma \ref{enough}. To this end, note first that, by assumption and Lemma \ref{equal gcds}, we have the equality $(\ida, NO_K) = (\idb, NO_K)$ of the greatest common divisors, which we denote by $\idd$; then, the $\sim_{NO_K}$-classes $[\ida]$ and $[\idb] \in DR_{NO_K}$ belong to the same component of the decomposition $DR_{NO_K} = \coprod_{\idd \mid NO_K} C_{NO_K / \idd}$. Again, as discussed in Lemma \ref{enough}, there exist $[\rho, s], [\rho, t] \in \widehat{O}_K \times_{\widehat{O}_K^*} (\Ad_{K,f}^* / K^*)$ such that $\idd = \rho \widehat{O}_K \cap K$ and $[\rho, s], [\rho, t]$ are mapped into $[\ida], [\idb] \in DR_{NO_K}$ respectively under the composition $\widehat{O}_K \times_{\widehat{O}_K^*} (\Ad_{K,f}^* / K^*) \simeq DR_K \twoheadrightarrow DR_{NO_K}$. Thus, to see the target congruence $\ida \sim_{NO_K} \idb$, it suffices to prove that $[s], [t] \in Gal(K^{ab} / K)$ define the same action on the ray class field $K_{NO_K / \idd}$ over $K$. 

We prove this by showing contradiction; suppose that $[s], [t]$ are not identical on $K_{NO_K / \idd}$. When this could be the case, as seen in Lemma \ref{enough}, there exists $a \in (N^{-1} \integer^2 / \integer^2) \cdot q_{\tau_K}(\rho) \simeq N^{-1} \idd / O_K$ such that $(f_a (\tau_K))^{[s]^{-1}} \neq (f_a (\tau_K))^{[t]^{-1}}$; take $a' \in N^{-1} \integer^2 / \integer^2$ so that $a' \cdot q_{\tau_K}(\rho) = a$. By construction, the ideal classes $[\ida], [\idb] \in DR_{NO_K}$ are the same respectively as the images of $[\rho, s], [\rho, t]$. Therefore, noting that $\widehat{f}_{a''}: DR_K \rightarrow \comp$ factors through $DR_K \twoheadrightarrow DR_{NO_K}$ for each $a'' \in N^{-1} \integer^2 / \integer^2$, we have $\widehat{f}_{a'} (\ida) = \widehat{f}_{a'} (\rho, s)$ and $\widehat{f}_{a'} (\idb) = \widehat{f}_{a'} (\rho, t)$. On the other hand, by assumption, we have:
\begin{equation}
 \widehat{f}_{a'}(\ida) = \widehat{f}_{a'}(\rho, s) = (f_a (\tau_K))^{[s]^{-1}} \neq (f_a(\tau_K))^{[t]^{-1}} = \widehat{f}_{a'} (\rho, t) = \widehat{f}_{a'}(\idb). 
\end{equation}
Thus $\psi_\ida \widehat{f}_{a'} (1) \neq \psi_\idb \widehat{f}_{a'} (1)$; however this contradicts to the assumption $\ida \equiv_{\Xi(N)} \idb$, whence $\psi_\ida \widehat{f}_{a'} = \psi_\idb \widehat{f}_{a'}$ must hold for $a' \in N^{-1} \integer^2 / \integer^2$. Therefore, $[s], [t]$ are identical on $K_{NO_K / \idd}$ and thus $\ida \sim_{NO_K} \idb$ as requested. 
\end{proof}

\begin{cor}
 We have an isomorphism of $\Lambda$-rings:
\begin{eqnarray}
 X_{\Xi(N)} &\simeq& X_{NO_K}. 
\end{eqnarray}
\end{cor}
\begin{proof}
 The galois object $X_{\Xi(N)}$ is dual to $I_K/ \equiv_{\Xi(N)}$, while $X_{NO_K}$ is dual to $I_K / \sim_{NO_K}$. Now as proved above, we know that $\equiv_{\Xi(N)}$ is equal to $\sim_{NO_K}$, hence the claim. 
\end{proof}

Finally we need the following:

\begin{lem}
\label{last lemma}
We have the equality: 
\begin{eqnarray}
X_{\Xi(N)} &=& K [\widehat{f}_a \mid a \in N^{-1} \integer^2 / \integer^2]. 
\end{eqnarray}
\end{lem}
\begin{proof}
It suffices to see that $\psi_\ida \widehat{f}_a = \widehat{f}_{a \cdot q_\tau (s_\ida)} ^{[s_\ida]}$, namely $\widehat{f}_a (\ida \idc) = (\widehat{f}_{a \cdot q_\tau(s_\ida)} (\idc))^{[s_\ida]}$ for each $\idc \in I_K$, since $a \cdot q_\tau (s_\ida) \in N^{-1} \integer^2/\integer^2$ for any $a \in N^{-1} \integer^2/\integer^2$ and the right-hand-side is closed under the action of $G_K^{ab}$. But this can be proved as $\widehat{f}_a (\ida \idc) = \widehat{f}_{a} (s_{\ida} s_\idc, s_\ida^{-1} s_\idc^{-1}) = (\widehat{f}_{a q_\tau(s_\ida)}(s_\idc, s_\idc^{-1}))^{[s_\ida]} = (\widehat{f}_{a q_{\tau}(s_\ida)}(\idc))^{[s_\ida]}$, where we choose $s_\ida \in \Ad_{K,f}^*$ so that $\ida = s_\ida \widehat{O}_K \cap K$, hence the claim. (To be more specific, the first equality is described in the proof of Lemma \ref{component of modular vector}; the second equality uses the $G_K^{ab}$-equivariance of $\widehat{f}_a$ and follows from the construction of $\widehat{O}_K \times_{\widehat{O}_K^*} (\Ad_{K,f}^*/K^*) \hookrightarrow \Gamma \backslash (M_2(\widehat{\integer}) \times \upper)$ in \S \ref{s2.2}; the last equality again uses the argument described in Lemma \ref{component of modular vector}.)
\end{proof}

This completes the proof of Theorem \ref{explicit modularity theorem}. 

\bibliographystyle{abbrv}
\bibliography{semigalois2B}
\end{document}